\documentclass[10pt,a4paper]{amsart}

    \usepackage[T1]{fontenc}
    \usepackage{amsmath}
    \usepackage{amssymb}
    \usepackage{amsthm}
    \usepackage{enumitem}
    \usepackage[left=4cm,right=4cm,top=3.5cm,bottom=3cm]{geometry}
    \usepackage{mathrsfs}
    \usepackage{verbatim}
    \usepackage{bbm}
    \usepackage{xcolor}
    \usepackage{mathtools}
    \usepackage[colorlinks=true, linktocpage=true, linkcolor=red!70!black, citecolor=green!50!black]{hyperref}
    \usepackage[normalem]{ulem}

    \usepackage[
      backend=biber,
      sorting=nyt,
      date=year,
      sortcites,
      style=alphabetic,
      abbreviate=true,
      giveninits=true,
      maxbibnames=10,
      maxcitenames=10,
      url=false,
      doi=true
    ]{biblatex}
    
    \DeclareSourcemap{
      \maps[datatype=bibtex]{
        \map[overwrite]{ 
          \step[fieldsource=shortjournal,fieldtarget=journaltitle]
        }
      }  
    }
    \bibliography{references}

    \providecommand{\eps}{\varepsilon}

    \newcommand{\norm}[1]{\lVert#1\rVert}

    \newcommand{\weakstarto}{\overset{*}{\rightharpoonup}}
    
    \renewcommand{\d}{\mathrm{d}}
    
    \AtBeginDocument{
      \let\div\relax
      \DeclareMathOperator{\div}{div}
    }

    \DeclareMathOperator{\cof}{cof}

    \DeclareMathOperator{\spt}{spt}

    \DeclareMathOperator{\dist}{dist}

    \DeclareMathOperator{\loc}{loc}

    \DeclareMathOperator*{\wslim}{w^*-lim}
    
    \def\XXint#1#2#3{\setbox0=\hbox{$#1{#2#3}{\int}$}\vcenter{\hbox{$#2#3$}}\kern-.5\wd0}

    \newcommand{\mres}{\mathbin{\vrule height 1.6ex depth 0pt width 0.13ex\vrule height 0.13ex depth 0pt width 1.3ex}}
    \DeclareMathOperator{\LL}{L}
    \DeclareMathOperator{\WW}{W}
    \DeclareMathOperator{\BB}{B}
    \DeclareMathOperator{\CC}{C}
    \DeclareMathOperator{\rmb}{b}
    \DeclareMathOperator{\rmc}{c}

    \newcommand{\FF}{\mathbf{F}}
    \newcommand{\GG}{\mathbf{G}}
    \renewcommand{\H}{\mathcal{H}}

    \newcommand{\mbdry}{\partial^{\mathrm{m}}}
    \newcommand{\rbdry}{\partial^{\mathrm{\ast}}}

    \newcommand{\N}{\ensuremath{\mathbb{N}}}

    \newcommand{\R}{\ensuremath{\mathbb{R}}}

    \theoremstyle{definition}
    \newtheorem{theorem}{Theorem}[section]
    \newtheorem*{theorem*}{Theorem}
    \newtheorem{corollary}[theorem]{Corollary}
    \newtheorem{definition}[theorem]{Definition}
    \newtheorem{example}[theorem]{Example}
    \newtheorem{lemma}[theorem]{Lemma}

    \newtheorem{remark}[theorem]{Remark}

    \numberwithin{equation}{section}

    \title[Admissibility criteria for normal traces and Cauchy fluxes]{Admissibility criteria for \\ normal traces and Cauchy fluxes}
    \author{Christopher Irving}
    \address{Christopher Irving,
    Department of Mathematics and Statistics, Georgetown University,
    3700 Reservoir Road NW, Washington, D.C., 20057, USA
    }
    \email{ci152@georgetown.edu}

    \author{Akram Sharif}
    \address{Akram Sharif, 
    Faculty of Mathematics, Institute of Geometry, Technische Universit\"at Dresden, 
    Zellescher Weg 12-14, 01069 Dresden
    }
    \email{akram.sharif@tu-dresden.de}
    \date{\today}

    \keywords{Cauchy fluxes, Divergence-measure fields, Minkowski content, finite perimeter, admissibility criteria}
    \subjclass[2020]{28C05, 74A10, 49Q15}

    \dedicatory{In memory of Miroslav \v{S}ilhav\'y, 
    for his foundational \\ contributions to mathematical continuum mechanics.}
\begin{document}

\begin{abstract}
  We compare notions of admissible surfaces for Cauchy fluxes, formulated as understanding when the normal trace of the underlying stress field can be represented by a measure.
  If this field is unbounded, the problem of admissibility is necessitated by the fact that the normal trace need not admit a measure representation on every regular surface, but only on ``almost all'' such surfaces.
  We compare an approach based on a precise majorant introduced by \v{S}ilhav\'y (Arch.\ Ration.\ Mech.\ Anal.\ 116.3 (1991)) with a Minkowski-type condition introduced by Chen, Torres and the first author (Arch.\ Ration.\ Mech.\ Anal.\ 249.6 (2025)) by showing that, under mild geometric conditions, the former condition implies the latter.
  We also show, by means of an explicit construction, that the latter admissibility condition can allow for arbitrary measure concentrations as a normal trace.
\end{abstract}

\maketitle

\section{Introduction}

\subsection{The problem of fluxes}
In mathematical continuum mechanics, a longstanding problem raised by Noll in \cite{Noll1959} concerns the rigorous derivation of stresses and formulation of balance laws when the underlying quantities 
{are highly irregular}.
{Of} particular interest is to rigorously derive generalizations of \emph{Cauchy's stress theorem}, which classically asserts that there exists a {stress field} $\FF$ such that a flux or {contact interaction} $\mathcal F(S)$ through a suitably smooth surface $S$ can be represented as
\begin{equation}
  \label{eq:Cauchy}
  \mathcal F(S) = \int_{S} \FF\cdot \nu_S \,d\mathcal H^{n-1},
\end{equation}
where $\nu_S$ is the unit normal field of $S$ and $\mathcal H^{n-1}$ denotes the $(n-1)$-dimensional Hausdorff measure. 
We will consider vector fields, however this theory extends to tensor fields by working componentwise.

Gurtin \& Martins \cite{GurtinMartins1976} proposed studying set functions $\mathcal F$ defined on suitable surfaces $S$, referred to as \emph{Cauchy fluxes}, for which the balance law
\begin{equation}\label{eq:balancelaw}
    \mathcal F(\partial E) = \mu(E),
\end{equation}
or a weak variant thereof, holds.
Subject to additivity and boundedness assumptions, we seek to find a field $\FF$ such that a representation of the form \eqref{eq:Cauchy} holds.
Here $\mu$ represents the \emph{production} of a quantity, which we assume to be a signed Radon measure to allow for singular concentrations.
For fluxes admitting the representation \eqref{eq:Cauchy} and satisfying the balance law, we have $\div \FF = \mu$, which naturally leads us to consider \emph{divergence-measure fields}.

In deriving a representation \eqref{eq:Cauchy} for a Cauchy flux, a central issue is to identify a suitable class of \emph{admissible surfaces} for which the flux is defined.
In particular, one seeks a suitable class $\mathcal B$ of \emph{bodies}, i.e.\ subsets $E\subset \R^n$, such that the flux is defined on \emph{parts} of the boundary $S \subset \partial E$.
Here it was proposed by Banfi \& Fabrizio \cite{BanfiFabrizio1979} to consider the class $\mathcal P$ of (normalized) \emph{sets of finite perimeter}, which provides a broad class of bodies supporting a Gauss-Green theorem (see \cite[Theorem~4.5.6]{Federer1996}).
This was adapted by Ziemer \cite{Ziemer1989}, and it was later established in \cite{ChenTorres2005,Silhavy2005} that the Gauss-Green theorem extends to bounded divergence-measure fields on sets of finite perimeter, which culminated in a characterization of Cauchy fluxes for bounded fields in \cite{ChenTorresZiemer2009}.

In general, however, one may encounter singular fields which are unbounded, such as those associated to concentrated loads as considered by Boussinesq \cite{Boussinesq1878} and Flamant \cite{Flamant1892}  (see \cite{Podio-Guidugli2005, Schuricht2008, Podio-GuidugliSchuricht2012}). 
This leads us to study fields which are merely locally integrable.
In this case, as observed in \cite{Silhavy1985}, we cannot hope to define the flux on every regular surface (see also \cite[Example~2.5]{Silhavy2009}), and only on an admissible class containing ``almost all surfaces''  depending on the flux, which we seek to specify.
In this direction \v{S}ilhav\'y \cite{Silhavy1985,Silhavy1991} introduced the admissibility condition
\begin{equation}\label{eq:h_admissible_intro}
  E \in \mathcal P \ \mbox{such that} \ \int_{\rbdry E} h \,\d \mathcal H^{n-1} < \infty,
\end{equation} 
where $h$ is a suitable \emph{majorant function} for the flux (see Definition~\ref{def:precise_majorant}) and $\rbdry E$ is the reduced boundary; such sets will be called \textit{$h$-admissible}.
Using this admissible class, Degiovanni, Marzocchi and Musesti established in \cite{DegiovanniEtAl1999} a one-to-one correspondence between unbounded divergence-measure fields and Cauchy fluxes, valid over sets satisfying \eqref{eq:h_admissible_intro}. Similar results were established in \cite{MarzocchiMusesti2001} and \cite{Schuricht2007a} for Cauchy interactions and contact interactions respectively.

Recently, Chen, Torres and the first author considered in \cite{ChenIrvingTorres2025} the setting where the underlying flux may be a vector measure.
In this generality, there is no apparent analogue of \eqref{eq:h_admissible_intro}, which led the authors to introduce a new admissibility condition, which in the context of integrable fields takes the following form:
\begin{equation}\label{eq:m_admissible_intro}
  U \subset \mathbb R^n \  \mbox{is open such that}\ \liminf_{\eps \downarrow 0} \frac1{\eps} \int_{U \cap B_{\eps}(\partial U)} m \,\d x < \infty.
\end{equation} 
Here $m$ is any majorant of $\FF$ in that $\lvert \FF \rvert \leq m$.
In the context of divergence-measure fields, a similar notion was introduced earlier by {\v{S}ilhav\'y} in \cite{Silhavy2009} as a sufficient condition for the distributional normal trace to be represented by a Radon measure,  and appears to be entirely different from the $h$-admissibility condition \eqref{eq:h_admissible_intro}.

The goal of the present work is to compare the two admissibility criteria \eqref{eq:h_admissible_intro} and \eqref{eq:m_admissible_intro}, which we study from the perspective of normal traces associated to divergence-measure fields.
Since we consider unbounded fields, the representation \eqref{eq:Cauchy} must be understood in a generalized sense through the \emph{distributional normal trace} of $\FF$ on $\partial E$, defined as 
\begin{equation}
  \langle \FF \cdot \nu, \phi \rangle_{\partial E} = \int_{E} \nabla \phi \cdot \FF \,\d x + \int_E \phi \,(\div \FF) \quad\mbox{for all } \phi \in \CC^1_{\rmb}(\Omega),
\end{equation}
\emph{i.e.}\ via the Gauss-Green formula.
Conditions \eqref{eq:h_admissible_intro}, \eqref{eq:m_admissible_intro} are  sufficient conditions to ensure this normal trace is represented by a Radon measure \cite{Silhavy1991,Silhavy2009,Schuricht2007a}, in which case we can understand \eqref{eq:Cauchy} as $\mathcal F(S) = \langle \FF \cdot \nu, \mathbbm{1}_S \rangle_{\partial E}$ for Borel subsets $S \subset \partial E$. 
For further properties of this distributional normal trace, we refer the reader to the aforementioned references and
\cite{ChenFrid1999,ChenFrid2003,SchurichtSchonherr2025,ChenComiTorres2019,Irving2025,ChenLiTorres2020,CrastaDeCicco2019,ComiCrastaDeCiccoMalusa2024}.

 \subsection{Main results}

Our first result shows that, under mild geometric assumptions on the surface, the $h$-admissibility condition \eqref{eq:h_admissible_intro} implies the Minkowski-type condition \eqref{eq:m_admissible_intro}, with an associated estimate. 
We defer the precise definitions of the notions involved.

\begin{theorem}
    \label{thm:intro_comparison}
    Let $\FF \in \LL^1(\mathbb R^n;\mathbb R^n)$ with precise majorant $h$ (as in Definition~\ref{def:precise_majorant}) and 
    let $E \subset \mathbb R^n$ be a set of finite perimeter such that $\partial E$ satisfies a uniform lower density condition (Definition~\ref{def:lower_density}).
    Then there exists $\gamma = \gamma(n,{E})>0$ such that
    \begin{equation}
        \label{eq:intro_F_bound}
          \liminf_{\eps\downarrow0} \frac{1}{2\eps} \int_{B_{\eps}({\partial E})} \lvert \FF \rvert \,\mathrm{d}x \leq \gamma\int_{{\partial E}} h \,\mathrm{d}\mathcal H^{n-1}.
    \end{equation}
\end{theorem}

To the best of our knowledge, this is the first result which compares the $h$-admissibility condition \eqref{eq:h_admissible_intro} and the Minkowski-type condition \eqref{eq:m_admissible_intro}, even for regular domains.

If $\FF \equiv 1$, then estimate \eqref{eq:intro_F_bound} reduces to a bound of the form
\begin{equation}\label{eq:intro_classic_minkowski}
  \liminf_{\eps \downarrow 0} \frac1{2\eps} \mathcal L^n(B_{\eps}(\partial E)) \leq \gamma \mathcal H^{n-1}(\partial E),
\end{equation} 
which relates the $(n-1)$-dimensional lower Minkowski content to the Hausdorff measure.
It is known that the uniform lower density condition is sufficient for such an estimate to hold (see \cite[Section 2.6]{AmbrosioEtAl2000}), and that already \eqref{eq:intro_classic_minkowski} can fail in its absence (see \cite{Kraft2016}).

On the other hand, we can show the converse fails by exhibiting a divergence-measure field $\FF$ with precise majorant $h$ and a regular surface $E$ such that
\begin{equation}
    \liminf_{\eps\downarrow0} \frac{1}{2\eps} \int_{B_{\eps}({\partial E})} \lvert \FF \rvert \,\mathrm{d}x < \infty, \quad\mbox{but}\quad \int_{{\partial E}} h \,\mathrm{d}\mathcal H^{n-1} = \infty.
\end{equation}
This will follow as a consequence of a construction of vector fields with a prescribed normal trace, which is our second main result.

\begin{theorem}\label{thm:intro_prescribed_trace}
    Let $\Omega \subset \mathbb R^n$ be a bounded Lipschitz domain and $\sigma \in \mathcal M(\partial \Omega)$.
    Then there exists a vector field $\FF \in \LL^1(\mathbb R^n; \mathbb R^n)$ with $\div \FF \in \LL^1(\R^n)$ such that 
    \begin{equation}\label{eq:intro_prescribed_minkowski}
        \limsup_{\eps \downarrow 0} \frac1{2\eps} \int_{B_\eps(\partial \Omega)} \lvert \FF \rvert \,\d x \leq C \lvert\sigma\rvert(\partial \Omega)<\infty
    \end{equation}
    and
    \begin{equation}\label{eq:intro_prescribed_trace}
        \langle \FF \cdot \nu, \cdot \rangle_{\partial\Omega} = \sigma.
    \end{equation}
    Moreover, if $\int_{\partial\Omega} \d\sigma=0$, then $\FF$ can be chosen to be divergence-free in a neighborhood of $\overline{\Omega}$.
\end{theorem} 

This provides a converse to \cite[Theorem~2.4]{Silhavy2009}, and implies that \eqref{eq:m_admissible_intro} allows for arbitrary measure concentrations of the normal trace, which is in contrast to fields satisfying the $h$-admissibility condition \eqref{eq:h_admissible_intro} (see \cite[Proposition~6.5]{Schuricht2007a}).

There are several existing ways to construct a vector field with prescribed trace as in \eqref{eq:intro_prescribed_trace}; the sharp trace space in this setting was shown to be $\WW^{-1,1}(\partial\Omega)$ in \cite{Irving2025}, and one can also use layer potential methods from \cite{MitreaTaylor2000} noting that $\mathcal M(\partial\Omega) \hookrightarrow \WW^{-\frac1p,p}(\partial\Omega)$ for $1 \leq p < \frac{n}{n-1}$.
 However, with both approaches it is unclear whether the constructed field also satisfies \eqref{eq:intro_prescribed_minkowski}.
We will instead employ a simpler potential construction in the half-space, and the constructed field will satisfy $\FF \in \LL^p(\mathbb R^n;\mathbb R^n)$ and $\div \FF \in \LL^p(\mathbb R^n)$ in the sharp range $1 \leq p < \frac{n}{n-1}$ (see \cite{ChenComiTorres2019}), which moreover improves with the fractional regularity of $\sigma$ (see Remark~\ref{rem:improvement_integrability}).

These results will then be applied to the theory of Cauchy fluxes associated to unbounded fields, comparing the notions given in \cite{Silhavy1991,DegiovanniEtAl1999} based on \eqref{eq:h_admissible_intro} and the recently introduced notion in \cite{ChenIrvingTorres2025} which uses \eqref{eq:m_admissible_intro} instead.
Our results suggest that the Minkowski-type condition \eqref{eq:m_admissible_intro} is more general, as on Lipschitz domains it is both implied by the $h$-admissibility condition \eqref{eq:h_admissible_intro}, and allows for fluxes with arbitrary measure concentrations.

\textbf{Organization:} 
We will collect basic notation and recall the notion of divergence-measure fields in Section~\ref{sec:admissibility}, along with known results concerning their normal traces.
Theorems~\ref{thm:intro_comparison} and \ref{thm:intro_prescribed_trace} will be established in Sections~\ref{sec:minkowski} and \ref{sec:prescribed} respectively.
Finally, in Section~\ref{sec:fluxes} we will apply our results to the theory of Cauchy fluxes, comparing in particular the formulations in \cite{DegiovanniEtAl1999} and \cite{ChenIrvingTorres2025}.

\section{Preliminaries}\label{sec:admissibility}

\subsection{Notation}
  Throughout the paper we work in $\mathbb R^n$ with $n \geq 2$, and denote by $\mathcal L^n$ and $\mathcal H^k$ the $n$-dimensional Lebesgue and $k$-dimensional Hausdorff measures respectively.
  For any set $E\subset \R^n$ and $\eps>0$, we denote the \textit{$\eps$-neighborhood of $E$} by
 \begin{equation}
    B_\eps(E) := E+B_\eps = \left\lbrace x+y \,:\, x\in E,\, y\in B_\eps \right\rbrace,
 \end{equation}
 where $B_{\eps} = B_\eps(0)$ denotes the $\eps$-ball centred at the origin in $\R^n$. 
  If $U \subset \mathbb R^n$ is open, taking $E=\partial U$ we obtain a two-sided $\eps$-neighborhood $B_{\eps}(\partial U)$, which compares to the interior and exterior neighborhoods $B_{\eps}(\partial U)\cap U$ and $B_{\eps}(\partial U)\setminus\overline U$ respectively.

 For a set $E \subset \mathbb R^n$ we denote the distance function of $E$ by
 \begin{equation}
    d_E(x) := \dist(x,E) = \inf_{y\in E} | x-y |,
 \end{equation} 
 which is $1$-Lipschitz and moreover satisfies $\lvert \nabla d_E(x)\rvert =1$ for $\mathcal L^n$-a.e.\ $x \in \mathbb R^n \setminus \overline E$.

We assume familiarity with sets of finite perimeter, see for instance \cite[Section 3.3]{AmbrosioEtAl2000}, however will briefly fix some relevant notation.
Given a Borel measurable set $E \subset \mathbb R^n$, the measure-theoretic interior and exterior are defined as
\begin{equation}
  E^1 = \{ x \in \mathbb R^n \colon \operatorname{dens}_E(x) = 1\}, \quad E^0 =\{ x \in \mathbb R^n \colon \operatorname{dens}_E(x) = 0\}
\end{equation} 
respectively, where
\begin{equation}\label{eq:density_function}
  \operatorname{dens}_E(x) = \lim_{r \downarrow 0}\frac{\mathcal L^n(E \cap B_r(x))}{\mathcal L^n(B_r(x))}.
\end{equation} 
Moreover, we denote the measure-theoretic boundary of $E$ as $\mbdry E = \mathbb R^n \setminus (E^1 \cup E^0)$ and say that $E$ is \emph{normalized} if $E = E^1$.
If $E$ is a set of finite perimeter, the reduced boundary will be denoted as $\rbdry  E$ and we have $\rbdry  E \subset \mbdry E \subset \partial E$ as well as $\mathcal H^{n-1}(\mbdry E \setminus \rbdry E)=0$.
Finally, we denote the generalized outer unit normal on $\rbdry E$ by $\nu_E$.

\subsection{Divergence-measure fields}

 For $1\leq p \leq \infty$ and $\Omega \subset \mathbb R^n$ an open set, a vector field $\FF\in \LL^{p}\left(\Omega;\R^{n}\right)$ is called \textit{divergence-measure field} if  
	\begin{equation*}
    \lvert\div \FF\rvert(\Omega) := \sup\left\lbrace\int_{\Omega} \FF\cdot \nabla\phi\,\d x \,:\,\phi\in \CC_{\rmc}^1(\Omega),\norm{\phi}_{\LL^{\infty}}\leq 1  \right\rbrace < \infty.
	\end{equation*}
  We denote the space of such fields as $\mathcal D\mathcal M^p\left(\Omega;\R^{n}\right)$.
  A direct consequence is that for all $\FF\in\mathcal D\mathcal M^p \left(\Omega;\R^{n}\right)$, there exists a unique $\mu\in\mathcal M(\Omega)$ such that
	\begin{equation} \label{eq:divergence measure fields}
    \int_{\Omega}\phi\,\d\mu = -\int_{\Omega} \FF\cdot \nabla \phi\,\d x \quad\mbox{for all $\phi \in \CC^1_{\rmc}(\Omega)$},
	\end{equation}
  and we write $\div \FF = \mu$. By approximation, one can show that \eqref{eq:divergence measure fields} extends to all $\phi\in \mathrm{Lip}_{\rmb}(\Omega)$, cf.\ \cite[Definition 5.1]{DegiovanniEtAl1999}.
  We also denote by $\mathcal{DM}_{\loc}^p(\Omega;\mathbb R^n)$ the space of fields $\FF$ such that $\FF \rvert_U \in \mathcal{DM}^p(U;\mathbb R^n)$ for all $U \Subset \Omega$.

  For a divergence-measure field $\FF\in\mathcal D\mathcal M^1\left(\Omega;\R^n\right)$ and a Borel set $E\subset \Omega$ the \textit{(weak) normal trace} of $\FF$ is the functional $\left<\FF\cdot \nu,\cdot\right>_{\partial E}:\mathrm{Lip}_{\rmb}(\Omega)\to \R$ given by
\begin{equation*}
    \left< \FF\cdot \nu,\phi\right>_{\partial E} := \int_{E} \phi\, \d(\div \FF) + \int_{E} \FF\cdot \nabla\phi \,\d x.
\end{equation*}
One can show (cf.\ \cite[Section 4]{ChenComiTorres2019}) that the support of $\langle\FF\cdot \nu,\cdot\rangle_{\partial E}$ as a distribution is supported on the topological boundary $\partial E$ and satisfies 
\begin{equation}
    \label{eq: dist normal trace}
    \left<\FF\cdot \nu,\cdot\right>_{\partial E} = \mathbbm{1}_E\mathrm{div}\, \FF - \mathrm{div}(\mathbbm{1}_E\FF) \quad \mbox{ as distributions in $\mathbb \Omega$},
\end{equation}
where $\mathbbm{1}_E$ denotes the characteristic function of $E$. 
Motivated by finding a representation of the form \eqref{eq:Cauchy}, of particular interest is when the normal trace \emph{is a measure} in that
there exists  $\sigma\in \mathcal{M}(\partial E)$ such that
\begin{equation}
    \label{eq: normal trace measure}
    \left<\FF\cdot \nu,\phi\right>_{\partial E} = \int_{\partial E} \phi \,\d\sigma, \quad \mbox{for all } \phi\in \mathrm{Lip}_{\rmb}(\Omega).
\end{equation}
We write $\sigma = (\FF \cdot \nu)_{\partial E}$ if this holds.
It follows from \eqref{eq: dist normal trace} that $\left<\FF\cdot \nu,\phi\right>_{\partial E}$ is a measure if and only if $\mathbbm{1}_E\FF\in\mathcal{DM}^1(\R^n,\R^n)$.
However, the normal trace space is strictly larger in general (see \cite[Theorem~1.7]{Irving2025}), which raises the question of when a representation of the form \eqref{eq: dist normal trace} holds.

\subsection{Admissibility conditions}

We collect some known conditions which ensure the normal trace of $\FF$ on $\partial E$ is represented by a measure.

\begin{definition}[Precise majorant $h$]\label{def:precise_majorant}
  Let $F \in \LL^1(\Omega)$ be non-negative with $\Omega \subset \mathbb R^n$ open, and fix a standard radial mollifier $\rho \in \CC^{\infty}_{\rmc}(B_1)$.
  Extending $F$ by zero to $\mathbb R^n$, we have $F \ast \rho_{\delta} \to \FF$ as $\delta \to 0$ in $\LL^1(\mathbb R^n)$, so we can choose a sequence $\delta_k \downarrow 0$ such that $F_k := F \ast \rho_{\delta_k}$ satisfies
  \begin{equation}
    \label{eq: bound mollification}
    \lVert F_{k+1}- F_{k} \rVert_{\LL^1(\Omega)} \leq 2^{-k}
  \end{equation} 
  for each $k\in\mathbb N$. Then setting 
  \begin{equation}
    F^{\ast} = \begin{cases}
      \displaystyle\lim_{k\to \infty} F_k(x) &\text{ if the limit exists}, \\
      0 &\text{ otherwise},
    \end{cases}
  \end{equation} 
  we then define
  \begin{equation}\label{eq:precise_h}
    h(x) := \lvert F_1(x) \vert + \sum_{k=1}^{\infty} \lvert F_{k+1}(x) - F_{k}(x)\rvert + \lvert F^{\ast}(x)\rvert
  \end{equation} 
  setting $h(x) = \infty$ if the series diverges. 
  This $h$ will be a \emph{precise majorant} of $F$, and $\FF^{\ast}$ will be the associated \emph{precise representative}. 
  It follows from \eqref{eq: bound mollification} that $h\in \LL^{1}(\Omega)$, hence proving that $h(x)<\infty$ for $\mathcal{L}^{n}$-a.e.\ $x\in\Omega$. 
  By construction, we also have that $|F_{k}|\leq h$ pointwise in $\Omega$ for all $k\in\mathbb{N}$.

  For vector fields $\FF = (F_1,\cdots,F_n) \in \LL^1(\Omega;\mathbb R^n)$, we decompose $F_i = F_i^+ - F_i^-$ into non-negative and non-positive parts, and let $h_i^{\pm}$ be the precise majorant for $F_i^{\pm}$, chosen along a common sequence $\delta_k \downarrow 0$. Then we define the precise majorant of $\FF$ as
  \begin{equation}\label{eq:precise_majorant_2}
    h := \sum_{i=1}^n ( h_i^+ + h_i^-).
  \end{equation}
\end{definition}

  In the context of Cauchy fluxes, this construction was first introduced by \v{S}ilhav\'y in \cite[Proposition~4.3]{Silhavy1991}, which was later adapted in \cite{DegiovanniEtAl1999,Schuricht2007a} and other works (see also \cite[Theorem~4.9]{Brezis2011}).
  Their construction is in fact more general; it suffices to take any function $h$ and a sequence $(\FF_k)$ of approximations satisfying
  \begin{align}
    \label{eq:h_property1}\lvert \FF_k \rvert(x) \leq h(x) \quad&\mbox{for all $k \in \mathbb N$ and  $x \in \Omega$}, \\
    \label{eq:h_property2}\lim_{k\to\infty} \FF_k(x) = \FF(x) \quad&\mbox{for all $x \in \Omega$ such that $h(x)<\infty$.}
  \end{align} 
  For the results presented in the sequel however, we will need the specific form defined above.

\begin{remark}
  While $\lvert \FF \rvert \leq h$ holds $\mathcal L^n$-a.e., the majorant may be considerably larger in general.
  Indeed let $\FF \in \LL^{\infty}(\Omega)$ be a bounded and discontinuous field. Then since the mollifications $\FF_k$ are smooth, it follows that $\FF_k$ converges pointwise $\mathcal L^n$-a.e.\,but not locally uniformly. Thus the summation in \eqref{eq:precise_h} does not converge locally uniformly near discontinuities, and hence $h$ is locally unbounded near points of discontinuity.
  Thus we conclude that $h \notin \LL^{\infty}(\Omega)$ necessarily.
\end{remark}

We record two sufficient conditions for the normal trace to be represented by a measure on the boundary.
The first is a result of {Schuricht} based on the precise majorant $h$ constructed above, which refines earlier results in \cite{Silhavy1991,DegiovanniEtAl1999}.

\begin{lemma}[{\cite[Proposition 6.5]{Schuricht2007a}}]\label{lem:schuricht_trace}
  Let $\Omega\subset \R^n$ be an open set, $\FF\in\mathcal{DM}^1(\Omega;\R^{ n})$ and let $h$ be a precise majorant of $\FF$ in that \eqref{eq:h_property1}, \eqref{eq:h_property2} hold. 
  Then, if $E \Subset \Omega$ is a normalized set of finite perimeter which is $h$-admissible in that
  \begin{equation}\label{eq:h_admissible}
    \int_{\rbdry E} h \,\d\mathcal H^{n-1} < \infty,
  \end{equation} 
  then the normal trace of $\FF$ on $\partial E$ is represented by a measure supported on $\mbdry  E$.
  Moreover there exists $g_E \in \LL^{\infty}(\Omega,\lvert \div \FF\rvert)$ such that 
  \begin{equation}
    (\FF \cdot \nu)_{\partial E} = -g_E \div \FF \mres \mbdry E + \FF^{\ast} \cdot \nu_E \,\mathcal{H}^{n-1}\mres \rbdry E.
  \end{equation} 
  Here $\FF^{\ast}(x) = \lim_{k \to \infty}\FF_k(x)$ is the pointwise limit of approximations, which is defined whenever $h(x)<\infty$ by \eqref{eq:h_property2}, and $g_E(x) = \operatorname{dens}_E(x)$ at all $x\in\Omega$ where the latter function, as defined in \eqref{eq:density_function}, exists.
\end{lemma}

The second result is a reformulated result of {\v{S}ilhav\'y}; a related result for Lipschitz domains can be found in \cite[Theorem~3.1]{ChenFrid2003},
\begin{lemma}[{\cite[Theorem~2.4]{Silhavy2009}}]\label{lem:silhavy_admissibility}
  If $U \Subset \Omega$ is an open set such that
  \begin{equation}\label{eq:silhavy_condition}
    \liminf_{\eps \downarrow 0} \frac1{\eps} \int_{U \cap B_{\eps}(\partial U)} \lvert \FF\cdot \nabla d_{\partial U}\rvert \,\d x < \infty
  \end{equation}
  where $d_{\partial U}(x) = \dist(x,\partial U)$,
  then the normal trace on $U$ is represented by a measure on $\partial U$, and there is a sequence $\eps_k \downarrow 0$ such that 
  \begin{equation}
    (\FF \cdot \nu)_{\partial U} = -\wslim_{\eps_k \to 0} \frac1{\eps_k} \nabla d_{\partial U} \cdot \FF \,\mathcal L^n \mres (U \cap B_{\eps}(\partial U)). 
  \end{equation} 
\end{lemma}
Since $\lvert \FF \cdot \nabla d_{\partial U}\rvert \leq \lvert \FF\rvert$ holds $\mathcal L^n$-a.e.\ in $\Omega$, condition \eqref{eq:silhavy_condition} is implied by \eqref{eq:m_admissible_intro} with the choice $m = \lvert \FF\rvert$.

We note that a similar condition was introduced earlier in \cite[Theorem~4.2]{Silhavy2005}, where if $E$ is a normalized set of finite perimeter, the condition
\begin{equation}\label{eq:silhavy_condition2}
  \liminf_{\rho \downarrow 0} \frac1{\omega_n \rho^n} \int_{\rbdry E} \int_{B_{\rho}(x)} \lvert \FF(y) \cdot \nu_E(x)\rvert \,\d y\,\d\mathcal H^{n-1}(x) < \infty
\end{equation} 
implies the normal trace of $\FF$ on $\rbdry E$ is a measure.
For Lipschitz domains, one can show that \eqref{eq:silhavy_condition2} is equivalent to having both $\Omega$ and $\mathbb R^n \setminus \overline\Omega$ satisfy \eqref{eq:silhavy_condition}.

\section{Minkowski-type estimate}\label{sec:minkowski}

In this section we will establish Theorem~\ref{thm:intro_comparison}, and in fact prove a more general statement. 
We will first specify the class of sets we will work with.

\begin{definition}\label{def:lower_density}
  Let $S \subset \mathbb R^n$ be a Borel measurable subset with locally finite $\mathcal H^{n-1}$-measure, in that $\mathcal H^{n-1}(S \cap K)<\infty$ for all $K \subset \mathbb R^n$ compact.
  We say $S$ satisfies the \emph{uniform lower density condition} (with respect to $\mathcal H^{n-1}$) if there exists $\gamma_0>0$ and $R_0>0$ such that
  \begin{align}\label{eq:uniform_lower}
      \mathcal H^{n-1}(S\cap B_r(x_0)) \geq\gamma_0r^{n-1},\quad \mbox{for all } x_0 \in S \mbox{ and } r\in(0,R_0).
    \end{align}
\end{definition}

Such sets are also called \emph{lower $(n-1)$-Ahlfors regular} in the literature.
It is well known that the boundaries of Lipschitz domains satisfy the uniform lower density estimate (see \cite[p.\,111]{AmbrosioEtAl2000}),
and further examples include sets with cusps, such as
\begin{equation*}
  S_\alpha = \left\lbrace (t,\lvert t\rvert^\alpha ) \, :\, t\in (-1,1)\right\rbrace,
  \quad\mbox{for each } \alpha \in (0,1).
\end{equation*}

\begin{remark}
If $S\subset \mathbb R^n$ has finite $\mathcal H^{n-1}$-measure and satisfies the uniform lower density condition, then using a covering argument (see the proof of Theorem~\ref{th: Estimate with two sided Minkowski}) we have
\begin{equation}\label{eq:minkowski_upperbound}
  \mathcal H^{n-1}(S) \geq C(n,\gamma_0) \limsup_{\eps \downarrow 0} \frac{\mathcal L^n(B_{\eps}(S))}{2\eps}.
\end{equation}
The $\limsup$ on the right hand side is the \textit{$(n-1)$-dimensional upper Minkowski content} of $S$. Moreover, it is well-known (see e.g.\ \cite[Theorem~2.104]{AmbrosioEtAl2000}) that if $S \subset \mathbb R^n$ is compact, countably $(n-1)$-rectifiable and satisfies the uniform lower density condition, then we have
\begin{equation*}
  \mathcal H^{n-1}(S) = \mathcal M^{n-1}(S) := \lim_{\eps \downarrow 0}  \frac{\mathcal L^n(B_{\eps}(S))}{2\eps}.
\end{equation*} 
The expression on the right hand side is the \textit{$(n-1)$-dimensional Minkowski content} of $S$. 
One-sided variants and related results can be found in \cite{AmbrosioColesantiVilla2008,Villa2009}. 
\end{remark}

The main result of this section is a generalization of the estimate \eqref{eq:minkowski_upperbound} to $\LL^1$-fields.

\begin{theorem}
    \label{th: Estimate with two sided Minkowski}
    Let $\FF \in \LL^1(\mathbb R^n;\mathbb R^n)$ with precise majorant $h$ and suppose $S \subset \mathbb R^n$ is Borel measurable with locally finite $\mathcal H^{n-1}$-measure, which satisfies the uniform lower density condition. Then there exists $\gamma>0$ such that
    \begin{equation}
        \int_{S} h \,\mathrm{d}\mathcal H^{n-1} \geq \gamma \liminf_{\eps\downarrow0} \frac{1}{2\eps} \int_{B_{\eps}(S)} \lvert \FF \rvert \,\mathrm{d}x.
    \end{equation}
\end{theorem}

We will later show the reverse estimate fails, cf.\ Corollary~\ref{thm:converse_fails}.
Before embarking on the proof, we present a simple example exhibiting the necessity of the lower density condition at every point on $S$, rather than e.g.\ on the reduced boundary when $S = \partial E$.

\begin{example}\label{eg:lowerdensity_2}
  Let $U = B_1(0) \setminus \{0\} \subset \mathbb R^2$ be the punctured unit disc, and define $\FF(x) = \chi(x)\frac{x}{\lvert x\rvert^2}$, where $\chi \in \CC^{\infty}_{\rmc}(B_1(0))$ is a cutoff such that $\chi \equiv 1$ in $B_{1/2}(0)$. 
  Then, $\FF \in \mathcal{DM}^1(\mathbb R^2)$ with $\div \FF = 2 \pi \delta_0 + \frac{x}{\lvert x\rvert^2} \cdot \nabla \chi$, and since $\FF$ is compactly supported in $B_1$, we can choose its precise majorant so that $h \equiv 0$ on $\partial B_1(0)$.
  Therefore we have
  \begin{equation}
    \int_{\partial U} h \,\d\mathcal H^{n-1} = 0 \quad\mbox{while}\quad \liminf_{\eps \downarrow 0} \frac1{\eps} \int_{B_\eps(\partial U)} \lvert \FF \rvert\,\d x 
    = 2\pi.
  \end{equation}
  In this case $\rbdry U = \partial B_1(0)$ where the lower density condition is satisfied, however the conclusion of Theorem \ref{th: Estimate with two sided Minkowski} fails with $S=\partial U$ since the lower density condition is not satisfied at $0 \in \partial U$.
\end{example}

\begin{lemma}\label{lem: general bound with h}
  Let $\FF \in \LL^1(\mathbb R^n;\mathbb R^n)$ with precise majorant $h$, and let $S \subset \mathbb R^n$ be a Borel measurable set with locally finite $\mathcal H^{n-1}$-measure satisfying the uniform lower density condition.
    Then there exists $\gamma>0$, a sequence $r_k \to 0$ and some index $N\in \N$ such that
    \begin{equation}
      \int_{S \cap B_{r_k}(x_0)} h \,\mathrm{d}\mathcal H^{n-1} \geq \frac{\gamma}{r_k} \int_{B_{r_k}(x_0)} \lvert \FF\rvert \,\mathrm{d}x  \quad \mbox{for all } x_0 \in S \mbox{ and } k \geq N.
    \end{equation}
\end{lemma}
\begin{proof}
  By splitting $\lvert \FF \rvert \leq \sum_{i=1}^n (\lvert \FF_i^+\rvert +\lvert \FF_i^-\rvert)$ and recalling the form of $h$ from \eqref{eq:precise_majorant_2}, we can reduce to the case when $\FF = F$ is a non-negative scalar function.
  In this case, there is a sequence $\delta_k \to 0$ such that $F_k = F\ast \rho_{\delta_k} \leq h$ pointwise. 
  Let $(r_k)_{k \in \mathbb N}$ be a sequence of positive radii to be determined.

  Let $x_0 \in S$, then up to a rigid transform we can assume that $x_0=0$.
  Then choosing $\kappa,\theta>0$ such that $\rho(y) > \theta$ for all $y \in B_{\kappa}$, we can estimate
    \begin{equation}\label{eq:hbound_lem_1}
        \begin{split}
            \int_{S \cap B_{r_k}} h \,\mathrm{d}\H^{n-1}
            &\geq \int_{S \cap B_{r_k}} F \ast \rho_{\delta_k} \,\mathrm{d}\H^{n-1} \\
            &\geq  \int_{S \cap B_{r_k}} \int_{B_{\delta_k}} \frac{F(x-y)}{\delta_k^n}\rho\left(\frac{y}{\delta_k}\right)\,dy \,\mathrm{d}\H^{n-1}(x) \\
            &\geq \frac{\theta}{\delta_k^n} \int_{S \cap B_{r_k}} \int_{B_{\kappa \delta_k}} F(x-y)\,\mathrm{d}y \,\mathrm{d}\H^{n-1}(x)  \\           
            &\geq \frac{\theta}{\delta_k^n} \int_{S \cap B_{r_k}} \int_{B_{\kappa \delta_k}(x)} F(y)\,\mathrm{d}y \,\mathrm{d}\H^{n-1}(x).
        \end{split}
    \end{equation}
    We choose $r_k = \frac{1}{4}\kappa \delta_k$, which ensures that $B_{r_k}\subseteq B_{\kappa \delta_k}(x)$ for all $x\in S \cap B_{r_k}$, and let $N\in\N$ such that $r_k < R_0$ holds for all $k \geq N$. With this choice, using \eqref{eq:uniform_lower}
 we can estimate
     \begin{align*}
        \int_{S \cap B_{r_k}} \int_{B_{\kappa \delta_k}(x)} F(y)\,\mathrm{d}y \,\mathrm{d}\H^{n-1}(x) 
        &\geq \int_{S \cap B_{r_k}} \int_{B_{r_k}} F(y)\,\mathrm{d}y \,\mathrm{d}\H^{n-1}(x) \\
        &= \mathcal{H}^{n-1}(S\cap B_{r_k}) \int_{B_{r_k}} F(y)\,\mathrm{d}y \\
        &\geq \gamma_0 r_k^{n-1} \int_{B_{r_k}} F(y)\,\mathrm{d}y 
     \end{align*}
     for all $k \geq N$.
     Combining this estimate with \eqref{eq:hbound_lem_1}, we deduce that
     \begin{align*}
         \int_{S \cap B_{r_k}} h \,\mathrm{d}\H^{n-1} & \geq \frac{\gamma_0\theta}{\delta_k}   \left(\frac{r_k}{\delta_k}\right)^{n-1} \int_{B_{r_k}} F(y)\,\mathrm{d}y = \frac{\gamma}{r_k} \int_{B_{r_k}} F(y)\,\mathrm{d}y,
     \end{align*}
     where $\gamma = \frac{\gamma_0\theta\kappa^{n}}{4^{n}}$.
\end{proof}

\begin{remark}
    For a general Borel-measurable subset $S \subset \mathbb R^n$ which is locally $\mathcal H^{n-1}$-finite, the proof of Lemma \ref{lem: general bound with h} shows that
    \begin{equation}\label{eq:local_q_bound}
        \int_{S \cap B_{r_k}} h \,\d\mathcal H^{n-1} \geq \frac{\theta \kappa^n}{4^n} \left( \frac{\mathcal H^{n-1}(S \cap B_{r_k}(x))}{r_k^{n-1}}\right) \frac1{r_k} \int_{B_{r_k}(x)} \lvert \FF \rvert \,\d x
    \end{equation}
    for all $x \in S$ and $k$ sufficiently large, where $r_k = \frac14 \kappa \delta_k$.
\end{remark}

\begin{proof}[Proof of {Theorem~\ref{th: Estimate with two sided Minkowski}}]
    By Lemma \ref{lem: general bound with h}, there exists $\gamma_1>0$ and $r_k \to 0$ such that
    \begin{equation}
        \int_{S \cap B_{r_k}(x_0)} h \,\d\mathcal{H}^{n-1} \geq \frac{\gamma_1}{r_k} \int_{B_{r_k}(x_0)} \lvert \FF \rvert \,\d x
    \end{equation}
    for all $x_0 \in S$ and $k$ sufficiently large.
    Now for each such $k$, we have $\{B_{r_k}(x) : x \in S\}$ is a covering of $B_{r_k/2}(S)$, so by Besicovitch's covering theorem there exists a constant $c_n$ and points $\{x_i^k\}_i \subset S$ such that
    \begin{equation}
        \mathbbm{1}_{B_{r_k/2}(S)} \leq \sum_{i} \mathbbm{1}_{B_{r_k}(x_i^k)} \leq c_n.
    \end{equation}
    Using this we can estimate
    \begin{equation}\label{eq:int_minkowski}
      \begin{split}
        \int_S h \,\d\mathcal H^{n-1} 
        &\geq \frac1{c_n} \sum_i\int_{S \cap B_{r_k}(x_i)} h \,\d\mathcal H^{n-1} \\
        &\geq \frac{\gamma_1}{c_n r_k} \sum_i \int_{B_{r_k}(x_i^k)} \lvert \FF \rvert \,\d x \\
        &\geq \frac{\gamma_1}{c_n r_k} \int_{B_{r_k/2}(S)} \lvert \FF \rvert \,\d x,
        \end{split}
    \end{equation}
    and so sending $k \to \infty$ gives
    \begin{equation}
        \int_S h \,\d \mathcal H^{n-1} \geq \frac{\gamma_1}{c_n} \liminf_{\eps \downarrow 0} \frac1{2\eps} \int_{B_{\eps}(S)} \lvert \FF \rvert \,\d x
    \end{equation}
    as required.
\end{proof}

\begin{corollary}
  Let $E \subset \mathbb R^n$ be a set of finite perimeter such that $\partial E$ satisfies the uniform lower density condition (Definition \ref{def:lower_density}).
  Then if $\FF \in \LL^1(\mathbb R^n;\mathbb R^n)$ is a vector field with precise majorant $h$ such that
  \begin{equation}
    \int_{\partial E} h \,\d \mathcal H^{n-1} < \infty,
  \end{equation} 
  then we have
  \begin{equation}
    \liminf_{\eps \downarrow 0} \frac1{2\eps} \int_{B_{\eps}(\partial E)} \lvert \FF \rvert \,\d x < \infty.
  \end{equation} 
\end{corollary}
In particular, if $\FF \in \mathcal{DM}^1(\Omega;\mathbb R^n)$, the admissibility condition \eqref{eq:silhavy_condition} in \v{S}ilhav\'y's result is implied by the $h$-admissibility condition \eqref{eq:h_admissible}.
However, as Example \ref{eg:lowerdensity_2} illustrates, we do not have an analogous result for $\partial^{\ast}E$. 

\begin{remark}
  If $E \subset \mathbb R^n$ is a normalized set of finite perimeter, arguing as in \eqref{eq:hbound_lem_1} with $S=\rbdry E$ in place of $S \cap B_{r_k}$ and using in particular that $\rho \geq \theta \mathbbm{1}_{B_{\kappa}}$, we can show that
  \begin{equation}
    \int_{\rbdry E} h \,\d \mathcal H^{n-1} \geq \frac{\theta}{\delta_k^n} \int_{\rbdry E} \int_{B_{\kappa\delta_k}(x)} \lvert \FF (y)\rvert \,\d y \,\d \mathcal H^{n-1}(x) \quad\mbox{for all $k \in \mathbb N$},
  \end{equation} 
  which implies the admissibility condition \eqref{eq:silhavy_condition2} from \cite[Theorem~4.2]{Silhavy2005}.
\end{remark}

We conclude this section with an example exhibiting the necessity of the uniform lower density condition already when $\FF \equiv 1$.

\begin{example}\label{eg:counterexample_kraft}
  In \cite[Example~1]{Kraft2016}, an open set $\Omega_0 \subset \mathbb R^2$ of finite perimeter with infinite $1$-Minkowski content is constructed, illustrating the necessity of the uniform lower density condition.
  While this example satisfies $\partial \Omega_0 = \rbdry \Omega_0 \cup (\{2\} \times (0,1))$, we will adapt this to obtain an example where the set is both normalized and open.
  Letting $D:= (0,2) \times (0,1)$, from \cite[Example~1]{Kraft2016} we choose $\Omega_0$ to be
  \begin{equation}
    \Omega_0 = \bigcup_{k \in \mathbb N} \bigcup_{j=1}^{N_k} B_{r_k}(x_{k,j}),
  \end{equation} 
  where we can set $r_k = \frac{16^{-k}}{4}, N_k=8^k$ and choose $x_{k,j}\in D$ as in \cite[Example~1 and Figure~5]{Kraft2016}.
  Then, we define
  \begin{equation}
    U = D \setminus \bigcup_{k \in \mathbb N} \bigcup_{j=1}^{N_k} \overline{B_{r_k}(x_{k,j})},
  \end{equation} 
  which is also a set of finite perimeter satisfying
  $  \partial U = \partial D \cup \rbdry \Omega_0 = \rbdry U \cup (\partial D \setminus \partial^{\ast}D)$, and since $\mathcal H^{1}(\partial D \setminus \partial^{\ast}D)=0$ it is normalized.
  We wish to show that
  \begin{equation}\label{eq:inner_minkowski_blowup}
    \liminf_{\eps \downarrow 0} \frac1{\eps} \mathcal L^2(U \cap B_{\eps}(\partial U)) = + \infty.
  \end{equation} 
  To see this, let $\eps_k = \frac12 \cdot 4^{-k} - \frac14 \cdot 16^{-k}$, then for each $\eps>0$ there exists $m\in \N$ for which $\eps_{m+1} \leq \eps < \eps_m$. 
  Then for each $k \leq m$, we have $B_{r_k+\eps}(x_{k,j}) \setminus \overline{B_{r_k}(x_{k,j})} \subset U$ and these annuli are pairwise disjoint, so we infer the lower bound
  \begin{equation}
    \begin{split}
    \mathcal L^2(U \cap B_{\eps}(\partial U)) 
    &\geq \sum_{k=0}^m \sum_{j=1}^{N_k} \mathcal L^2(B_{r_k+\eps}(x_{k,j}) \setminus \overline{B_{r_k}(x_{k,j})} ) \\
    &\geq \sum_{k=0}^m N_k \cdot \pi ((r_k+\eps)^2- r_k^2) \geq 2\pi \eps^2 \sum_{k=0}^m N_k,
  \end{split}
  \end{equation} 
  so the calculations in \cite[(26)]{Kraft2016} imply the estimate
  \begin{equation}
    \frac1\eps \mathcal L^2(U \cap B_{\eps}(\partial U)) \geq \frac{\pi}{14} \left( \frac1{2\sqrt{\eps}}-1 \right),
  \end{equation} 
  which blows up as $\eps \downarrow 0$, establishing \eqref{eq:inner_minkowski_blowup}.
\end{example}

\section{Structure of admissible traces}\label{sec:prescribed}

We will now prove Theorem~\ref{thm:intro_prescribed_trace}, and use it to show that the converse to Theorem~\ref{thm:intro_comparison} fails.
More precisely, we establish the following:

\begin{theorem}\label{thm:prescribed_trace}
    Let $\Omega \subset \mathbb R^n$ be a bounded Lipschitz domain and $\sigma \in \mathcal M(\partial \Omega)$.
    Then there exists a vector field $\FF$ on $\mathbb R^n$ such that $\FF \in \LL^p(\mathbb R^n; \mathbb R^n)$ with $\div \FF \in \LL^p(\R^n)$ for all $1 \leq p < \frac{n}{n-1}$, 
    \begin{equation}\label{eq:prescribed_minkowski}
        \limsup_{\eps \downarrow 0} \frac1{2\eps} \int_{B_\eps(\partial \Omega)} \lvert \FF \rvert \,\d x \leq C \lvert\sigma\rvert(\partial \Omega)<\infty,
    \end{equation}
    and
    \begin{equation}\label{eq:prescribed_trace}
        \langle \FF \cdot \nu, \cdot \rangle_{\partial\Omega} = \sigma.
    \end{equation}
    Moreover, if $\int_{\partial\Omega} \d\sigma=0$, then $\FF$ can be chosen to be divergence-free in a neighborhood of $\overline\Omega$.
\end{theorem}

\begin{remark}\label{rem:improvement_integrability}
  In light of \cite[Theorem 4.1]{ChenComiTorres2019} (see also \cite[Theorem 4.2]{Silhavy2005}), the range of $p$ is optimal. 
  We can, however, obtain an improvement in integrability if the measure $\sigma$ satisfies an additional fractional regularity constraint. 
  More precisely, in the setting of Theorem~\ref{thm:prescribed_trace}, if for $p \geq \frac{n}{n-1}$ we have in addition that
  \begin{equation}\label{eq:sigma_fractional}
    \sigma \in \WW^{-\frac1p,p}(\partial\Omega),
  \end{equation} 
  then the constructed field $\FF$ lies in $\LL^p(\mathbb R^n,\mathbb R^n)$ with $\div \FF \in \LL^p(\mathbb R^n)$.
  As it is of tangential interest, we will defer the proof of this assertion to Appendix \ref{sec:appendix_improvement}.

  We note that a sufficient condition for this fractional regularity \eqref{eq:sigma_fractional} is given in \cite[Theorem~4.7.4]{Ziemer1989}, namely if for some $m > n -1- \frac1{p-1}$ we have
  \begin{equation}\label{eq:m-regular}
    \lvert \sigma \rvert(B_r(x)) \leq c_m r^m \quad\mbox{for all $x \in \partial\Omega$, $r >0$},
  \end{equation} 
  then $\sigma \in \WW^{-\frac1p,p}(\partial\Omega)$ and hence the constructed field $\FF$ lies in $\LL^p(\Omega)$.
  Similar conditions were considered in \cite[Proposition~6.1]{Silhavy2005} for the prescribed divergence problem.
\end{remark}

The proof of Theorem~\ref{thm:prescribed_trace} will involve an explicit potential construction in the half-space, which is an adaptation of Gagliardo's extension operator for Sobolev traces. We isolate this construction in the following lemma.

\begin{lemma}\label{lem:G_potential}
  Given a standard mollifier $\rho$ on $\mathbb R^{n-1}$, we define
  \begin{equation}\label{eq:potential}
    \Phi_j(y) =  -\operatorname{sgn}(y_n)\frac{y_j}{\lvert y_n\rvert^n} \rho\bigg(\frac{y'}{\lvert y_n\rvert}\bigg) \quad\mbox{for } j= 1,\cdots,n
  \end{equation} 
  for all $y = (y',y_n) \in \mathbb R^n$ with $y_n\neq0$.
  Using this potential we define for $\sigma \in \mathcal M(\mathbb R^{n-1})$ the field
  \begin{equation}\label{eq:G_potential}
    \GG(y',y_n) := (\Phi(\cdot,y_n) \ast \sigma)(y') = \int_{\mathbb R^{n-1}} \Phi(y'-z',y_n)\,\d\sigma(z'),
  \end{equation} 
  on $\mathbb R^n$. Then $\GG$ satisfies the following properties:
  \begin{enumerate}[label=(\roman*)]
    \item\label{item:G_integrability}For each $1 \leq p < \frac{n}{n-1}$ and $T>0$, we have $\GG \in \LL^p(\mathbb R^{n-1} \times (-T,T))$ along with the estimate
      \begin{equation}
        \lVert \GG \rVert_{\LL^p(\mathbb R^{n-1} \times (-T,T))} \leq c(n,p) T^{\frac{n}p+1-n} \lvert \sigma\rvert(\mathbb R^{n-1}).
      \end{equation} 

    \item\label{item:G_divfree} $\GG$ is divergence-free in $\mathbb R^n$,

    \item\label{item:G_trace} $\langle \GG \cdot \nu, \cdot \rangle_{\partial\mathbb R^n_+} = \sigma = - \langle \GG \cdot \nu, \cdot \rangle_{\partial\mathbb R^n_-}$.
  \end{enumerate}
\end{lemma}
\begin{proof}
  Observe that, for each $y_n \neq 0$, by the change of variables $z' = y'/y_n$ we have
  \begin{equation}\label{eq:Phi_slicewise_integrable}
      \int_{\mathbb R^{n-1}} \lvert \Phi(y',y_n) \rvert^p \,\d y' = \lvert y_n\rvert^{-(n-1)(p-1)}\int_{\mathbb R^{n-1}} \lvert \Phi(z',1) \rvert^p \,\d z' < \infty,
  \end{equation}
  Since $y_n \mapsto \lvert y_n\rvert^{-(n-1)(p-1)}$ is locally integrable on $\mathbb R$ provided $1 \leq p < \frac{n}{n-1}$, we have $\Phi \in \LL^p(\mathbb R^n \times (-T,T))$ for such $p$ and $T>0$.
  Now if $y_n>0$, we can compute
  \begin{align}
    \partial_{y_j}\Phi_j(y',y_n) &= -\frac1{y_n^n} \rho\bigg(\frac{y'}{y_n}\bigg) - \frac{y_j}{y_n^{n+1}} \partial_{y_j}\rho\bigg(\frac{y'}{y_n}\bigg) \quad\mbox{for each } 1 \leq j \leq n-1, \\
    \partial_{y_n}\Phi_n(y',y_n) &= \frac{n-1}{y_n^n} \rho\bigg(\frac{y'}{y_n}\bigg) + \sum_{i=1}^{n-1} \frac{y_i}{y_n^{n+1}} \partial_{y_i}\rho\bigg(\frac{y'}{y_n}\bigg),
  \end{align}
  from which it follows that $\div \Phi = 0$ in $\mathbb R^n_+$, and the calculation on $\mathbb R^n_-$ is analogous.

  If we define $\GG$ via \eqref{eq:G_potential}, by Young's convolution inequality and \eqref{eq:Phi_slicewise_integrable} we have
  \begin{equation}\label{eq:G_integrability_estimate}
    \begin{split}
    &\int_{-T}^T \int_{\mathbb R^{n-1}} \lvert \GG(y',y_n)\rvert^p \,\d y'\,\d y_n \\
    &\quad\leq \lvert \sigma \rvert(\mathbb R^{n-1}) \int_{-T}^T \lvert y_n\rvert^{-(n-1)(p-1)} \,\d y_n \int_{\mathbb R^{n-1}} \lvert \Phi(z',1)\rvert^p\,\d z',
    \end{split}
  \end{equation} 
  from which we infer \ref{item:G_integrability}.
  Since $\Phi$ is smooth and divergence-free on $\mathbb R^n \setminus \partial\mathbb R^n_+$, it follows that $\div \GG = 0$ on $\mathbb R^n \setminus \partial\mathbb R^n_+$ also,
  and the fact that it is divergence free on $\mathbb R^n$ will follow from \ref{item:G_trace}.
  To show this, let $v \in \CC_{\rmc}^{\infty}(\mathbb R^n)$.
  Noting that $\rho_{y_n} \ast \sigma \weakstarto \sigma$ weakly${}^{\ast}$ as measures as $y_n \to 0$ on $\mathbb R^{n-1}$, applying \cite[Theorem.\,2.4]{Silhavy2009} we have
  \begin{equation}
    \begin{split}
      \langle \GG \cdot \nu, v \rangle_{\partial\mathbb R^n_+}
      &= \lim_{\eps \downarrow 0} \frac1{\eps} \int_0^{\eps} \int_{\mathbb R^{n-1}} v(y) \GG(y',y_n) \cdot \mathrm{e}_n \,\d y' \,\d y_n \\
      &= -\lim_{\eps \downarrow 0} \frac1{\eps} \int_0^{\eps} \int_{\mathbb R^{n-1}} v(y) (\rho_{y_n} \ast \sigma)(y') \,\d y' \,\d y_n \\
      &= \int_{\mathbb R^{n-1}} v(y',0) \,\d \sigma(y'),
    \end{split}
  \end{equation} 
  noting that $\dist(y,\partial\mathbb R^n_+) = y_n$ on $\mathbb R^n_+$.
  Similarly, since $\Phi_n(y',-y_n) = \Phi_n(y',y_n)$, we have
  \begin{equation*}
    \langle \GG \cdot \nu, v \rangle_{\partial \mathbb R^n_-} =\lim_{\eps \downarrow 0} \frac{1}{\eps} \int_0^{\eps} \int_{\mathbb R^{n-1}} v(y',-y_n) \GG(y',-y_n) \cdot (-\mathrm{e}_n) \,\d y'\,\d y_n = - \langle \GG \cdot \nu, v \rangle_{\partial \mathbb R^n_+},
  \end{equation*} 
  thereby establishing \ref{item:G_trace}, and as a consequence \ref{item:G_divfree}.
\end{proof}

\begin{proof}[Proof of Theorem~{\rm\ref{thm:prescribed_trace}}]
  Let $\sigma \in \mathcal M(\partial \Omega)$. The proof will involve several steps.

  \noindent\textbf{Step 1}.\ (Reduction to the half-space): Let $r>0$ and $x_1,\cdots,x_N \in \partial\Omega$ such that $\{B_r(x_i)\}_{i=1}^N$ covers $\partial\Omega$, and such that for each $i$, there exists a bi-Lipschitz flattening map
    \begin{equation}
        \varphi_i \colon B_{2r}(x_i) \to B_2(0)
    \end{equation}
    satisfying
    \begin{gather}\label{eq:flattening_1}
      \varphi_i(B_r(x_i)) \subset B_{1}(0) \subset \varphi_i(B_{2r}(x_i)), \\
        \label{eq:flattening_2}
      \varphi_i(\Omega \cap B_{2r}(x_i)) \subset B_2(0) \cap \mathbb R^n_+, \quad
        \varphi_i(B_{2r}(x_i)\setminus\overline\Omega) \subset B_2(0) \cap \mathbb R^n_-,\\
        \label{eq:flattening_3}
      \varphi_i(\partial\Omega \cap B_{2r}(x_i)) \subset B_2(0) \cap \partial\mathbb R^n_+.
    \end{gather}
    In what follows, we abbreviate $B_s^{\pm} = B_s(0) \cap \mathbb R^n_{\pm}$ and $\Gamma_s = B_s(0) \cap \partial\mathbb R^n_+$ for each $s>0$.
    We also denote $\psi_i = \varphi_i^{-1}$, which is well-defined on $B_{1}(0)$, and let $L>0$ be such that each $\varphi_i$ and $\psi_i$ are $L$-Lipschitz.
    Moreover, we can choose $\varphi_i$ to be orientation preserving in that $\det \varphi_i > 0$ a.e. in $B_{2r}(x_i)$.
    Then, using $\eqref{eq:flattening_3}$, we have
    \begin{equation}\label{eq:distance_flattened_equiv}
        L^{-1} \lvert\varphi_i(x)_n\rvert \leq \dist(x,\partial\Omega) \leq L \lvert\varphi_i(x)_n\rvert \quad\mbox{for all $x \in B_{2r}(x_i)$},
    \end{equation}
    and hence it follows for all $0 < \eps < \frac1L$ that
    \begin{equation}\label{eq:flattened_annular}
      \left\{ x \in B_1 : \lvert x_n \rvert < \tfrac{\eps}L \right\} \subset \varphi_i\left(B_{2r}(x_i) \cap B_{\eps}(\partial\Omega)\right) \subset \left\{ x \in B_2 : \lvert x_n \rvert < L \eps\right\}.
    \end{equation}
    We also let $\{\chi_i\}_{i=1}^N \subset \CC^{\infty}_{\rmc}(\mathbb R^n)$ be a partition of unity subordinate to the covering $\{B_r(x_i)\}_{i=1}^N$ of a neighborhood of $\partial\Omega$.

    Extending $\sigma$ by zero to be a measure on $\mathbb R^n$ we set,
    \begin{equation}\label{eq:Gi_1}
      \tilde\sigma_i = (\varphi_i)_{\#}(\sigma \mres B_{2r}(x_i)) \in \mathcal M(\Gamma_1),
    \end{equation} 
    which we view as a measure on $\mathbb R^{n-1}$.
    By applying Lemma~\ref{lem:G_potential} to each $\tilde{\sigma}_i$, we obtain divergence-free fields $\GG_i \in \LL^1(B_1)$ satisfying
    \begin{equation}
      \langle \GG_i \cdot \nu, \,\cdot\, \rangle_{\partial B_1^+} \mres \Gamma_1 = \widetilde\sigma_i \quad\mbox{as measures on $\Gamma_1$}.
    \end{equation} 
    Also for each $\eps \in (0,\frac1L)$, applying the estimate in \ref{item:G_integrability} with $p=1$ and $T=\eps$, we obtain the estimate
    \begin{equation}\label{eq:Gi_2}
    \frac1{\eps} \int_{B_1} \mathbbm{1}_{\R^{n-1}\times(-\eps,\eps)}(y) \lvert \GG_i(y) \rvert \,\d y \leq C \lvert \tilde\sigma_i \rvert(\Gamma_1) < \infty \quad\mbox{for all $0<\eps<\tfrac1L$}.
    \end{equation}

    \noindent\textbf{Step 2.}\ (Transform and patching): We now transform the constructed fields $\GG_i$ back along the flattening map $\varphi_i$ and patch them using a partition of unity.
    For each $1 \leq i \leq N$, we define $\FF_i$ on $B_r(x_i)$ via the Piola transform as
    \begin{equation}\label{eq:Fi_defn}
      \FF_i= (\det \nabla\varphi_i) \, (\nabla\varphi_i)^{-1} \cdot \GG_i  \circ \varphi_i \quad\mbox{in } B_r(x_i),
    \end{equation} 
    which equivalently satisfies
    \begin{equation}\label{eq:Gi_defn}
      \GG_i =   (\cof \nabla \psi_i)^{\top} \cdot \FF_i \circ \psi_i\quad\mbox{in }\varphi_i(B_r(x_i))\subset B_1(0).
    \end{equation} 
    Since $\GG_i \in \LL^p(B_1)$ for each $1 \leq p < \frac{n}{n-1}$, we have $\FF_i \in \LL^p(B_r(x_i))$.
    By noting that $\div \cof \nabla\psi_i = 0$ (the Piola identity) and that $\nabla (\FF_i \circ \psi_i) = (\nabla \psi_i)^{\top} \cdot (\nabla \FF_i) \circ \psi_i$, this field satisfies
   \begin{equation}
     \begin{split}
     \div \GG_i = \mathrm{Tr}( (\cof\nabla\psi_i)^{\top} \cdot (\nabla \psi_i)^{\top} \cdot (\nabla\FF_i)\circ\psi_i ) = (\det \nabla \psi_i) (\div \FF_i) \circ \psi_i
   \end{split}
    \end{equation} 
    in $\varphi_i(B_r(x_i)) \setminus \{ y_n =0 \}$. 
    Hence, by the area formula (see e.g.\ \cite[Theorem.\,2.71, (2.47)]{AmbrosioEtAl2000}), if $v \in \CC^1_{\rmc}(B_r(x_i))$ we have
    \begin{equation}
      \int_{\Omega \cap B_r(x_i)} v(x) \div \FF_i(x) \,\d x = \int_{B_1^+} (v \circ \psi_i)(y) \div \GG_i(y) \,\d y.
    \end{equation} 
    Similarly using \eqref{eq:Gi_defn} and noting that $(\cof \nabla \psi)^{-\top} = (\det \nabla\psi_i)^{-1} (\nabla \psi_i)^{\top}$, we can compute that
    \begin{equation}
      \begin{split}
        (\nabla v \cdot \FF_i)\circ \psi_i 
        &= [(\nabla \psi_i)^{-\top} \cdot \nabla (v \circ \psi_i)] \cdot [(\cof \nabla \psi_i)^{-\top} \cdot \GG_i]\\
        &= \frac1{\det\nabla\psi_i} \, \nabla(v \circ \psi_i) \cdot \GG_i,
      \end{split}
    \end{equation} 
    so by the area formula we have
    \begin{equation}
      \int_{\Omega \cap B_r(x_i)} \nabla v(x) \cdot \FF_i(x) \,\d x = \int_{B_1^+} \nabla(v \circ \psi_i(y)) \cdot \GG_i(y) \,\d y.
    \end{equation} 
    Therefore
    \begin{equation}\label{eq:Fi_trace}
      \begin{split}
        \langle \FF_i \cdot \nu, v\rangle_{\partial\Omega}
        &=  \int_{\Omega \cap B_r(x_i)} \nabla v \cdot \FF_i \,\d x + \int_{\Omega \cap B_r(x_i)} v \div \FF_i \,\d x \\
        &= \int_{B_1^+} \nabla (v \circ \psi_i) \cdot \GG_i \,\d y + \int_{B_1^+} v \circ \psi_i \div \GG_i \,\d y  \\
        &= \int_{\Gamma_1} v \circ \psi_i \,\d \widetilde\sigma_i = \int_{\partial\Omega \cap B_r(x_i)} v \,\d \sigma.
      \end{split}
    \end{equation} 
    Arguing analogously on the complement of $\Omega$, we can show that
    \begin{equation}
       \int_{B_r(x_i)\setminus\overline\Omega} \nabla v \cdot \FF_i \,\d x + \int_{B_r(x_i)\backslash\overline\Omega} v \div \FF_i \,\d x = -\int_{\partial\Omega \cap B_r(x_i)} v \,\d\sigma,
    \end{equation}
    so it follows in particular that $\div \FF_i = 0$ in $B_r(x_i)$.
    Using \eqref{eq:flattened_annular} and \eqref{eq:Gi_2}, we can also verify for $0 < \eps < \frac1{L^2}$ that
    \begin{equation}\label{eq:Fi_limsup}
      \begin{split}
      \frac1{\eps}\int_{B_{\eps}(\partial\Omega)\cap B_r(x_i)} \lvert \FF_i \rvert \,\d x
      &\leq \frac1{\eps} \int_{-L \eps}^{L \eps} \int_{\mathbb R^{n-1}} \lvert \GG_i \rvert (\det \nabla\psi_i)^{-1} \,\d y'\,\d y_n \\
      &\leq c \lvert \widetilde\sigma_i\rvert(\Gamma_1) \leq c \lvert \sigma\rvert(\partial\Omega \cap B_r(x_i)).
    \end{split}
    \end{equation} 

    Now, recall that $\{\chi_i\}$ was a partition of unity subordinate to $\{B_r(x_i)\}$ of $\partial\Omega$. Extending each $\chi_i\FF_i$ by zero to $\mathbb R^n$, we define
    \begin{equation}
      \FF = \sum_{i=1}^N \chi_i \FF_i,
    \end{equation} 
    which is 
    compactly supported and $\LL^p$-integrable for each $1 \leq p < \frac{n}{n-1}$.
    Also since each $\FF_i$ is divergence-free in $B_r(x_i)$, we have
    \begin{equation}
      \label{eq: div FF}
      \div \FF = \sum_{i=1}^N \nabla\chi_i  \cdot \FF_i \in \LL^p(\mathbb R^n)
    \end{equation} 
    for each $1 \leq p < \frac{n}{n-1}$.
    Using \eqref{eq:Fi_limsup} we also have
    \begin{equation}
      \frac1{\eps} \int_{B_{\eps}(\partial\Omega)} \lvert \FF \rvert \,\d x \leq C \lvert \sigma \rvert(\partial\Omega) \quad\mbox{for all $0 < \eps < \tfrac1{L^2}$},
    \end{equation} 
    which implies \eqref{eq:prescribed_minkowski}.
    To show the trace is attained, let $v \in \CC^1_{\rmc}(\mathbb R^n)$ and observe that $\chi_i v \in \CC^1_{\rmc}(B_r(x_i))$ for each $i$.
    Then, 
    using the linearity of the trace and \eqref{eq:Fi_trace}, we have
    \begin{equation}
      \begin{split}
        \langle \FF \cdot \nu, v \rangle_{\partial\Omega}
        &= \sum_{i=1}^N \langle (\chi_i\FF_i) \cdot \nu, v \rangle_{\partial\Omega}\\
        &= \sum_{i=1}^N \left[\int_{\Omega \cap B_r(x_i)} \chi_i \nabla v \cdot \FF_i \,\d x + \int_{\Omega \cap B_r(x_i)}  v \,\div (\chi_i\FF_i) \,\d x   \right] \\
        &= \sum_{i=1}^N \left[\int_{\Omega \cap B_r(x_i)} \nabla(\chi_i v) \cdot \FF_i \,\d x + \int_{\Omega \cap B_r(x_i)} \chi_i v \,\div \FF_i \,\d x   \right] \\
        &= \sum_{i=1}^N \int_{\partial\Omega\cap B_r(x_i)} \chi_i v \,\d \sigma 
        = \int_{\partial\Omega} v \,\d \sigma,
      \end{split}
    \end{equation} 
    where in the last line we used that $\sum_{i=1}^N \chi_i \equiv 1$ on $\partial\Omega$,
    verifying \eqref{eq:prescribed_trace}.

    \textbf{Step 3}.\ (Divergence-free correction):
    Suppose that $\int_{\partial\Omega} \d \sigma = 0$.
    By the previous step, there exists a compactly supported vector field $\FF$ satisfying \eqref{eq:prescribed_minkowski} and \eqref{eq:prescribed_trace}, along with the estimate
  \begin{equation}
    \lVert \FF \rVert_{\LL^p(\mathbb R^n)} + \lVert \div \FF \rVert_{\LL^p(\mathbb R^n)} \leq C_p \lvert\sigma \rvert(\partial\Omega) \quad\mbox{for all } 1 \leq p < \frac{n}{n-1},
  \end{equation} 
  which follows from Lemma~\ref{lem:G_potential}\ref{item:G_integrability} and \eqref{eq: div FF}.
  Since we have 
  \begin{equation}
    \int_{\Omega} \div \FF \,\d x = \langle \FF \cdot \nu, \mathbbm{1}_{\mathbb R^n}\rangle_{\partial\Omega} = \int_{\partial\Omega} \d\sigma = 0,
  \end{equation}
  we can apply Bogovskii's operator from \cite{Bogovskii1979} to obtain a field $\mathbf{v}_1 \in \WW_0^{1,1}(\Omega;\mathbb R^n)$ satisfying
  \begin{equation}
    \div \mathbf{v}_1 = \div \FF \quad\mbox{in } \Omega
  \end{equation} 
  and the estimate
  \begin{equation}
    \lVert \mathbf{v}_1 \rVert_{\WW^{1,p}(\Omega)}  \leq C \lVert \div \FF \rVert_{\LL^{p}(\mathbb R^n)} \leq C_p \lvert \sigma\rvert(\partial\Omega) \quad\mbox{for all } 1 < p < \frac{n}{n-1}.
  \end{equation}
  To verify that $\mathbf{v}_1$ satisfies \eqref{eq:prescribed_minkowski}, we consider the bi-Lipschitz mappings $\varphi_i = \psi_i^{-1} \colon B_{2r}(x_i) \to B_2(0)$ from the previous two steps.
  Since $\{\chi_i\}_{i=1}^N$ is a partition of unity of a neighborhood of $\partial\Omega$ subordinate to $\{B_r(x_i)\}_{i=1}^N$, there is $\eps_0>0$ such that $\sum_i \chi_i \equiv 1$ on $B_{\eps_0}(\partial\Omega)$.
  Then for $0<\eps<\eps_0/L$ and each $1 \leq i \leq N$ we have
  \begin{equation}
    \frac1{\eps} \int_{\Omega \cap B_{\eps}(\partial\Omega) \cap B_r(x_i)} \chi_i\lvert \mathbf{v}_1\rvert \,\d x \leq \frac1{\eps} \int_0^{L\eps} \int_{\mathbb R^{n-1}} \chi_i \circ \psi_i \lvert \mathbf{v}_1 \circ \psi_i \rvert (\det \nabla\psi_i)^{-1}\,\d y.
  \end{equation} 
  Moreover, since $\mathbf{v}_1 \circ \psi \in \WW^{1,1}(B_1^+;\mathbb R^n)$ vanishes on $\Gamma_1$, for $\mathcal L^{n-1}$-a.e.\ $y' \in \mathbb R^{n-1}$ we can estimate
  \begin{equation}
    \begin{split}
      \frac1{\eps} \int_0^{L \eps} \lvert\mathbf{v}_1 \circ \psi_i \rvert(y',y_n) \,\d y_n 
      &\leq \frac1{\eps} \int_0^{L \eps} \int_0^{y_n} \lvert \nabla \mathbf{v}_1 \circ \psi_i \rvert(y',t) \lvert \nabla \psi_i\rvert(y',t) \,\d t \,\d y_n \\
      &\leq L \int_0^{L \eps} \lvert \nabla \mathbf{v}_1 \circ \psi_i \rvert(y',t) \lvert \nabla \psi_i \rvert(y',t)\,\d t.
    \end{split}
  \end{equation}
  Combining the above estimates gives
  \begin{equation}
    \begin{split}
      \frac1{\eps} \int_{\Omega\cap B_{\eps}(\partial\Omega)\cap B_r(x_i)} \chi_i\lvert \mathbf{v}_1\rvert \,\d x
      &\leq L^2\int_0^{L\eps}\int_{\mathbb R^{n-1}} \chi_i \circ \psi_i \lvert \nabla \mathbf{v}_1 \circ \psi_i \rvert  (\det \nabla \psi_i)^{-1} \,\d y \\
      &= L^2 \int_{\Omega  \cap B_{L^2\eps}(\partial\Omega)\cap B_{2r}(x_i)} \chi_i \lvert \nabla\mathbf{v}_1 \rvert \,\d x,
    \end{split}
  \end{equation} 
  which vanishes in the limit $\eps \downarrow 0$.
  Therefore summing over $i$ we infer that
  \begin{equation}\label{eq:v1_admissible}
    \limsup_{\eps \downarrow 0} \frac1{\eps} \int_{\Omega \cap B_{\eps}(\partial\Omega)} \lvert \mathbf{v}_1\rvert \,\d x = 0.
  \end{equation} 

  For the exterior correction, let $R>0$ such that $\overline{\Omega} \subset B_R$, which exists since $\Omega$ is bounded,
  and let $\{C_j\}_{j=1}^J$ denote the connected components of $B_R \setminus \overline{\Omega}$.
  Then for each $1 \leq j \leq J$, fix $x_j \in C_j$, $r_j>0$ such that $\overline{B}_{r_j}(x_j) \subset C_j$, and put $\widetilde{C}_j = C_j \setminus \overline{B}_{r_j}(x_j)$.
  Then setting $A = B_R \setminus \bigcup_{j=1}^J \overline{B}_{r_j}(x_j)$, we have $A \setminus \overline\Omega = \bigcup_{j=1}^J \widetilde C_j$, and each $\widetilde C_j$ is a Lipschitz domain.

  We then choose $\boldsymbol{\zeta}_j \in \CC^{\infty}(\overline C_j;\mathbb R^n)$ such that $\boldsymbol{\zeta}_j = 0$ in a neighbourhood of $\partial C_j$, and that $\int_{\widetilde{C_j}} \div \boldsymbol{\zeta}_j \neq 0$; for instance take $\boldsymbol{\zeta}_j(x) = \eta_j(x)\frac{x-x_j}{\lvert x - x_j\rvert^{n}}$ with $\eta_j \in \CC^{\infty}_{\rmc}(C_j)$ a cutoff such that $\eta_j \equiv 1$ in a neighbourhood of $\overline B_{r_j}(x)$.
  Then for each $j$, there exists $k_j \in \mathbb R$ for which
  \begin{equation}\label{eq:defining_k_j}
    \int_{C_j} \div \FF - k_j \div \boldsymbol{\zeta}_j \,\d x = 0.
  \end{equation} 
By applying Bogovskii's result to each of the fields $\FF - k_j\boldsymbol{\zeta}_j$ on $\widetilde C_j$, we obtain a vector field $\mathbf{v}_2 \in \WW^{1,1}_0(A\setminus\overline\Omega;\mathbb R^n)$ such that
  \begin{equation}
    \div \mathbf{v}_2 = \div \FF - \sum_j k_j \div \boldsymbol{\zeta}_j \quad\mbox{on }A \setminus \overline\Omega = \bigcup_{j=1}^J \widetilde{C}_j.
  \end{equation} 
  Additionally, by \eqref{eq:defining_k_j} and using the estimate from Bogovskii's result, we deduce that $\mathbf{v}_2$ satisfies the estimate
  \begin{align}
          \begin{split}
    \lVert \mathbf{v}_2\rVert_{\WW^{1,p}(A\setminus\overline\Omega)} 
    &\leq c \sum_j \lVert \div \FF - k_j \div\boldsymbol{\zeta}_j \rVert_{\LL^p(A\setminus\overline\Omega)} \\
    &\leq c (\lVert \FF\rVert_{\LL^p(\mathbb R^n)}+ \lVert \div \FF \rVert_{\LL^p(\mathbb R^n)}) \leq c \lvert\sigma\rvert(\partial\Omega)
    \end{split}
  \end{align}
  for each $1 < p < \frac{n}{n-1}$.
  Arguing analogously as in the proof of \eqref{eq:v1_admissible} we infer that
  \begin{equation}
    \limsup_{\eps \downarrow 0} \frac1{\eps} \int_{B_{\eps}(\partial\Omega)\setminus\overline\Omega} \lvert\mathbf{v}_2\rvert \,\d x = 0.
  \end{equation} 
  To conclude we set
  \begin{equation}
    \mathbf{v} = \mathbf{v}_1 \mathbbm{1}_{\Omega} + \mathbf{v}_2 \mathbbm{1}_{A \setminus \overline{\Omega}} + \sum_j k_j  \boldsymbol{\zeta}_j,
  \end{equation} 
  and verify that the field $\FF - \mathbf{v}$ satisfies all of the claimed properties in $A$, which is a neighbourhood of $\overline\Omega$.
\end{proof}

\begin{remark}[Globally divergence-free fields]\label{rem:globally_divfree}
  In the setting of Theorem~\ref{thm:prescribed_trace}, let $\{C_j\}$ denote the connected components of $\mathbb R^n \setminus \overline{\Omega}$.
  If $\sigma$ satisfies $\int_{\partial C_j} \,\d \sigma = 0$ for each $j$, then we can choose $\FF$ to be divergence-free in $\mathbb R^n$.
  This occurs in particular when $\mathbb R^n \setminus \overline\Omega$ is connected and $\int_{\partial\Omega} \d\sigma = 0$.

  Indeed from the construction in Steps 1 and 2 of the proof, we obtain a compactly supported field $\FF \in \LL^1(\mathbb R^n)$, so choose $R>0$ such that $\overline\Omega \cup \spt(\FF) \subset B_R(0)$.
  Recalling that $C_j$ are the connected components of $B_R \setminus \overline\Omega$, since
  $\int_{C_j} \div \FF \,\d x = 0$ for each $j$, we can apply Bogovskii's result on each $C_j$ to obtain a $\mathbf{v}_2 \in \WW^{1,1}_0(B_R \setminus \overline\Omega;\mathbb R^n)$ such that $\div\mathbf{v}_2 = \div \FF$. From here the result follows by an analogous argument as in Step 3.
\end{remark}

\begin{remark}
  This result is in contrast to Lemma~\ref{lem:schuricht_trace} of Schuricht, which implies that if $\FF \in \mathcal{DM}^1(\Omega)$ and $U \subset \Omega$ is $h$-admissible with respect to the precise majorant of $\FF$, then the normal trace of $\FF$ on $\partial U$ is a measure satisfying
  \begin{equation}
    \lvert (\FF \cdot \nu)_{\partial U}\rvert \ll \mathcal H^{n-1} \mres \rbdry U + \lvert \div \FF \rvert \mres \partial U.
  \end{equation} 
  In comparison, Theorem~\ref{thm:prescribed_trace} shows the admissibility condition of \v{S}ilhav\'y, namely that
  \begin{equation}
    \liminf_{\eps \downarrow 0} \frac1{2\eps} \int_{B_{\eps}(\partial\Omega)} \lvert \FF \rvert\,\d x < \infty
  \end{equation} 
  does not imply any additional structure on the normal trace $(\FF \cdot \nu)_{\partial U}$, beyond that it is represented by a measure.
\end{remark}

As a consequence, we can show that the converse statement to Theorem~\ref{th: Estimate with two sided Minkowski} is false. 

\begin{corollary}\label{thm:converse_fails}
  Let $\Omega \subset \mathbb R^n$ be a bounded Lipschitz domain.
  Then there exists a divergence-free field $\FF \in \LL^1(\mathbb R^n)$ with precise majorant $h$ 
  such that
  \begin{equation}\label{eq:converse_fails}
    \int_{\partial\Omega} h \,\d\mathcal{H}^{n-1} = \infty \quad\mbox{whereas}\quad
    \lim_{\eps \downarrow 0} \frac1{2\eps} \int_{B_{\eps}(\partial\Omega)} \lvert \FF \rvert \,\d x < \infty.
  \end{equation} 
\end{corollary}

\begin{proof}
  Choose $\mu \in \mathcal M(\partial\Omega)$ such that $\int_{\partial\Omega} \,\d \mu = 0$ and $\mu \not\ll \mathcal H^{n-1}$; for instance we can pick two distinct points $x_1,x_2 \in \partial\Omega$ lying on the same connected component and take $\mu = \delta_{x_2} - \delta_{x_1}$.
  Then by Theorem~\ref{thm:prescribed_trace} and Remark~\ref{rem:globally_divfree}, we obtain a divergence-free field $\FF \in \LL^1(\mathbb R^n)$ satisfying \eqref{eq:converse_fails}${}_2$ and such that $(\FF \cdot \nu)_{\partial \Omega} = \mu$.
  However, $\Omega$ is not $h$-admissible, as otherwise Lemma~\ref{lem:schuricht_trace} would imply that $\mu \ll \mathcal H^{n-1}$, which contradicts our standing assumption that $n \geq 2$.
\end{proof}

\section{Applications to Cauchy fluxes}\label{sec:fluxes}

Equipped with the results obtained in the previous sections, we consider two formulations of Cauchy fluxes for $\mathcal{DM}^1$ fields and discuss how they differ.
We will start by recalling these two notions in the literature, using a unified notation.

\subsection{Definitions of the Cauchy flux}
The first formulation is from \cite{DegiovanniEtAl1999}, which is based on prior works of \cite{Silhavy1985,Silhavy1991}. A related formulation can be found in \cite{Schuricht2007a}.

\begin{definition}\label{def:h_admissible}
  Let $h$ be a non-negative integrable function on $\Omega$.
  The class of \emph{$h$-admissible} sets, denoted by $\mathcal{P}_h$, consists of sets $E\subset\Omega$ of finite perimeter that are normalized, in that $E^1 = E$, and satisfy
  \begin{equation}
    \int_{\rbdry E} h \,\d\mathcal H^{n-1} < \infty.
  \end{equation} 
  In addition, given a measure $\sigma$ on $\Omega$, we denote by $\mathcal P_{h,\sigma}$ the class of sets $E \in \mathcal{P}_h$ for which $\lvert\sigma\rvert(\mbdry E)=0$ also holds.
\end{definition}

The following is an equivalent formulation of Cauchy fluxes as defined in \cite{DegiovanniEtAl1999}.

\begin{definition}\label{def:cauchyflux1}
  Let $\Omega \subset \mathbb R^n$ be open, $\sigma$ be a signed Radon measure on $\Omega$, and $h$ be a non-negative integrable function on $\Omega$.
  A \emph{Cauchy flux} in the sense of \cite{DegiovanniEtAl1999} in $\Omega$ is a set function $\mathcal Q$ defined on pairs $(S,E)$ with $E \in \mathcal P_{h,\sigma}$ and $S \subset \mbdry E$ any Borel subset, which satisfies the following:
  \begin{enumerate}[label=(\alph*)]
    \item\label{item:localisation_1} If $E_1, E_2 \in \mathcal P_{h,\sigma}$ and $S \subset \mbdry E_1 \cap \mbdry E_2$ such that $\nu_{E_1} = \nu_{E_2}$ holds $\mathcal H^{n-1}$-a.e.\,on $S$, then $\mathcal Q_{E_1}(S) = \mathcal Q_{E_2}(S)$.

    \item If $E \in \mathcal P_{h,\sigma}$ and $S_1, S_2 \subset \mbdry E$ are disjoint, then $$\mathcal Q_E(S_1 \cup S_2) = \mathcal Q_E(S_1) + \mathcal Q_E(S_2).$$

    \item\label{item:blaw1} The balance law $\mathcal Q(\mbdry  E) = \sigma(E)$ holds for all $E \in \mathcal P_{h,\sigma}$.

    \item\label{item:hbound_1} For any $E \in \mathcal P_{h,\sigma}$ and any Borel subset $S \subset \mbdry E$, we have the estimate
      \begin{equation*}
        \lvert \mathcal Q_{E}(S)\rvert \leq \int_S h \,\d\mathcal H^{n-1} 
      \end{equation*} 
  \end{enumerate}
\end{definition}

Compared to \cite[Definition~3.9]{DegiovanniEtAl1999}, the notable difference is the addition of \ref{item:localisation_1}, which is implicit in the original work as the authors work with \emph{oriented surfaces} $(S,\nu_S)$.
Moreover, in place of \ref{item:blaw1} the authors only assume a \emph{weak balance law} in that $\lvert \mathcal{Q}(\partial E)\rvert \leq \mu(E)$ for some non-negative measure $\mu$ on $\Omega$.

\begin{theorem}[\cite{DegiovanniEtAl1999}]\label{thm:flux_equiv_1}
  Let $\mathcal Q$ be a Cauchy flux in the sense of Definition~\ref{def:cauchyflux1}.
  Then, there exists a unique field $\FF \in \mathcal{DM}^{1}(\Omega)$ such that
  \begin{equation}\label{eq:flux_as_trace1}
    \mathcal{Q}_E(S) = (\FF \cdot \nu)_{\partial E}(S)  \quad\mbox{for all $E \in \mathcal P_{h,\sigma}$ and $S \subset \mbdry E$ a Borel set}.
  \end{equation} 
  Conversely, if $\FF \in \mathcal{DM}^1(\Omega)$, letting $\sigma = \div \FF$ and $h$ be a precise majorant of $\FF$ as in Definition~\ref{def:precise_majorant}, then \eqref{eq:flux_as_trace1} defines a Cauchy flux as in Definition~\ref{def:cauchyflux1}.
\end{theorem}

\begin{remark}
  In fact, in \cite{DegiovanniEtAl1999} the authors obtain the representation
  \begin{equation}
    \mathcal{Q}_E(S) = \int_{S}\FF^* \cdot \nu_E \,\d\mathcal H^{n-1},
  \end{equation} 
  where $\FF^{\ast}$ is the associated precise representative, which more closely resembles \eqref{eq:Cauchy} from the introduction.
With the form \eqref{eq:flux_as_trace1} however, by using \cite[Proposition~6.5]{Schuricht2007a} it is possible to define the flux $\mathcal Q$ on pairs $(S,E)$ with $E \in \mathcal P_h$ and extend the recovery result \eqref{eq:flux_as_trace1} to this class.
  In this case, one must replace condition \ref{item:hbound_1} with
  \begin{enumerate}[label=(\alph*${}^{\prime}$)]
    \setcounter{enumi}{3}
    \item \label{eq:dprime}  For any $E \in \mathcal P_{h}$ and any Borel subset $S \subset \mbdry E$, we have the estimate
      \begin{equation*}
        \lvert \mathcal Q(S)\rvert \leq \int_S h \,\d\mathcal H^{n-1} + \lvert \sigma \rvert(S).
      \end{equation*} 
  \end{enumerate}
  This requires extending the uniqueness result \cite[Theorem~4.9]{DegiovanniEtAl1999}; since this is not the main focus of the present work, we only sketch the details. More precisely we show that if $\mathcal Q^1$, $\mathcal Q^2$ are two Cauchy fluxes such that $\mathcal Q^1_E(S)=\mathcal Q^2_E(S)$ for all $E \in \mathcal P_{h,\sigma}$ and $S \subset \mbdry E$, then this also holds for all $E \in \mathcal P_h$.
  Here one can run steps IV and V of the proof of Theorem~4.9 (in \cite{DegiovanniEtAl1999}) with general $M \subset \mathcal P_h$, as these steps rely on Theorem~4.7(c) which also holds in $\mathcal P_h$ using \ref{eq:dprime}, and Lemma~4.8 which generalises by applying \cite[Theorem~2.2.2]{Federer1996} with the measure $\mathcal H^{n-1} + \lvert\sigma\rvert$ to obtain a decomposition $S_0 = T_0 \cup (\bigcup_{k=1}^{\infty}T_k)$ such that $\mathcal H^{n-1}(T_0) = \lvert\sigma\rvert(T_0) = 0$.
\end{remark}

The second formulation was recently introduced in \cite{ChenIrvingTorres2025}; this applies in the generality of measure-valued fields, however we restrict to $\LL^1$-integrable fields in the present discussion.
Given a non-negative function $m \in \LL^1(\Omega)$, the admissible class is given by
\begin{equation}
  \mathcal O_m = \bigg\{ U \subset \Omega : \mbox{$U$ is open, } \liminf_{\eps \downarrow 0} \frac1{\eps} \int_{U \cap B_{\eps}(\partial U)} m \,\d x < \infty \bigg\}.
\end{equation} 

\begin{definition}\label{def:cauchyflux2}
  Let $\Omega \subset \mathbb R^n$ be open, $\sigma$ be a signed measure on $\Omega$ and $m \in \LL^1(\Omega)$ be non-negative.
  A \emph{Cauchy flux} in the sense of \cite{ChenIrvingTorres2025} is a set function $\mathcal F$ defined on pairs $(S,U)$ with $U \in \mathcal O_m$ and $S \subset \partial U$ any Borel set, which satisfies the following:
  \begin{enumerate}[label=(\alph*)]
    \item If $U,V \in \mathcal O_m$ and $A \subset \Omega$ is open such that $A \cap U = A \cap V$, then $\mathcal F_U(A \cap \partial U) =\mathcal F_V(A \cap \partial V)$.

    \item If $U \in \mathcal O_m$ and $S_1,S_2 \subset \partial U$ are disjoint Borel subsets, then $\mathcal F_U(S_1 \cup S_2) = \mathcal F_U(S_1) + \mathcal F_U(S_2)$.

    \item\label{item:blaw2} The balance law $\mathcal F(\partial U) = \sigma(U)$ holds for all $U \in \mathcal O_m$.

    \item\label{item:upperbound_2} For any $U \in \mathcal O_{m}$ and any Borel set $S \subset \partial U$, we have
      \begin{equation*}
        \lvert \mathcal F_U(S) \rvert \leq m_U^{n-1}(S)
      \end{equation*} 
      for any limiting measure
      \begin{equation*}
        \frac1{\eps_k} m \,\mathcal L^n \mres (U \cap B_{\eps_k}(\partial U)) \xrightharpoonup{\ \ast\ } m_U^{n-1} \quad\mbox{as $\eps_k \downarrow 0$}.
      \end{equation*} 
  \end{enumerate}
\end{definition}

\begin{theorem}[\cite{ChenIrvingTorres2025}]\label{thm:flux_equiv_2}
  Let $\mathcal F$ be a Cauchy flux in the sense of Definition~\ref{def:cauchyflux2}.
  Then there exists a unique field $\FF \in \mathcal{DM}^1(\Omega)$ such that
  \begin{equation}\label{eq:flux_trace2}
    \mathcal F_U(S) = (\FF \cdot \nu)_{\partial U}(S) \quad\mbox{for all $U \in \mathcal O_m$ and Borel set $S \subset \partial U$}.
  \end{equation} 
  Conversely if $\FF \in \mathcal{DM}^1(\Omega)$ and we let $\sigma = \div \FF$ and $m = \lvert \FF\rvert$, then \eqref{eq:flux_trace2} defines a Cauchy flux as in Definition~\ref{def:cauchyflux2}.
\end{theorem}

\subsection{Comparison of the admissibility conditions}
The two notions of Cauchy fluxes introduced in the previous section build on fundamentally different classes of admissible sets $\mathcal P_h$ and $\mathcal O_m$, where one considers normalized sets of finite perimeter, whereas the other considers general open sets.
We have seen in Example~\ref{eg:counterexample_kraft} that already for $h \equiv m \equiv 1$, there exists $U \in \mathcal P_h$ such that $U \notin \mathcal O_m$.
However Theorem~\ref{th: Estimate with two sided Minkowski} implies that, roughly speaking, if $U \Subset \Omega$ is sufficiently regular then $U \in \mathcal P_h$ implies that $U \in \mathcal O_m$.
The following theorem makes this assertion precise, and also gives a quantitative bound.

\begin{theorem}\label{thm:flux_comparison}
  Let $\Omega \subset \mathbb R^n$ be an open set, let $\FF \in \mathcal{DM}^1(\Omega)$ with precise majorant $h$, and put $m = \lvert \FF\rvert$.
  Suppose $U \Subset \Omega$ is an open set of finite perimeter such that $\mathcal H^{n-1}(\partial U) = \mathcal H^{n-1}(\rbdry U)$ and $\partial U$ satisfies the lower density condition (Definition~\ref{def:lower_density}).
  Then if $U \in \mathcal P_h$, we have $U \in \mathcal O_{m}$ and
  there is a limiting measure $m_U^{n-1}$ as in Definition~\ref{def:cauchyflux2}\ref{item:upperbound_2} such that
  \begin{equation}\label{eq:hm_measure_comparison}
    m_U^{n-1} \leq c  h\,\mathcal H^{n-1} \mres \partial U \quad\mbox{as measures in $\Omega$},
  \end{equation} 
  where $c = c(n,\gamma_0)>0$.
\end{theorem}

\begin{proof}
  By Theorem~\ref{th: Estimate with two sided Minkowski}, we have $U \in \mathcal{O}_m$, and moreover by
  \eqref{eq:int_minkowski} from the proof of Theorem~\ref{th: Estimate with two sided Minkowski}, we have
  \begin{equation}
    \frac2{r_k} \int_{B_{r_k/2}(\partial U)} m \,\d x \leq C < \infty
  \end{equation} 
  for all $k$ sufficiently large.
  Therefore by passing to a (non-relabeled) subsequence, we obtain a limiting measure
  \begin{equation}
    \frac2{r_k} m\,\mathcal L^n \mres (U \cap B_{r_k/2}(\partial U)) \xrightharpoonup{\ \ast\ } m_U^{n-1} \quad\mbox{as $k \to \infty$}.
  \end{equation} 
  To show \eqref{eq:hm_measure_comparison} holds with this $m_U^{n-1}$, let $\phi \in \CC(\overline\Omega)$ be non-negative and fix $\eps>0$. By continuity of $\phi$, there exists $0<\delta<\dist(U,\partial\Omega)$ such that for all $x \in \partial U$ and $y \in B_{\delta}(x)$ we have $\lvert\phi(x) - \phi(y) \rvert<\eps$.

  By Lemma~\ref{lem: general bound with h}, there is $r_k \downarrow 0$ and $\gamma>0$ such that for $k$ sufficiently large, we have
  \begin{equation}\label{eq:hlemma_estimate}
    \int_{\partial U \cap B_{r_k}(x_0)} h \,\d\mathcal H^{n-1} \geq \frac{\gamma}{r_k} \int_{U \cap B_{r_k}(x_0)} m \,\d x \quad\mbox{for all $x_0 \in \partial U$}.
  \end{equation}
  Let $k_0 \in \mathbb N$ be such that $r_{k} \leq \delta$ and \eqref{eq:hlemma_estimate} holds for all $k \geq k_0$. Then, for such $k$ and each $x_0 \in \partial U$, multiplying both sides of \eqref{eq:hlemma_estimate} by $\phi(x_0)$ and using the continuity estimate we have
  \begin{equation}\label{eq:hestimate_phi}
    \begin{split}
      &\int_{\partial U \cap B_{r_k}(x_0)} \phi\:\! h \,\d\mathcal{H}^{n-1} 
      -\frac{\gamma}{r_k} \int_{U \cap B_{r_k}(x_0)} \phi \:\! m \,\d x \\ 
       &\quad\geq \int_{\partial U \cap B_{r_k}(x_0)} (\phi(x_0)-\eps) h \,\d\mathcal H^{n-1} + \frac{\gamma}{r_k} \int_{U \cap B_{r_k}(x_0)} (-\phi(x_0)-\eps)m \,\d x\\
      &\quad= - \eps \left[\int_{\partial U \cap B_{r_k}(x_0)} h \,\d\mathcal H^{n-1} + \frac{\gamma}{r_k} \int_{U \cap B_{r_k}(x_0)} m \,\d x\right] \\
      &\qquad+ \phi(x_0) \left[\int_{\partial U \cap B_{r_k}(x_0)} h \,\d\mathcal H^{n-1} - \frac{\gamma}{r_k} \int_{U \cap B_{r_k}(x_0)} m \,\d x\right] \\
      &\quad\geq - 2\eps \int_{\partial U \cap B_{r_k}(x_0)} h \,\d\mathcal H^{n-1},
    \end{split}
  \end{equation} 
  where we used \eqref{eq:hlemma_estimate} twice in the last line.
  Now since $\{B_{r_k}(x_0) : x_0 \in \partial U\}$ is a covering of $B_{r_k/2}(\partial U)$ by the Besicovitch covering theorem there exist finitely many points $\{x_i\}_{i=1}^N$ such that $\{B_{r_k}(x_i)\}$ cover $B_{r_k/2}(\partial U)$, where at most $c_n$ balls overlap at any given point.
  Therefore summing over \eqref{eq:hestimate_phi} with $x_i$ in place of $x_0$ we obtain
  \begin{equation}\label{eq:phi_estimate_error}
    \begin{split}
    \int_{\partial U} \phi\:\!h\,\d\mathcal{H}^{n-1} 
    &\geq \frac1{c_n} \sum_{i=1}^N \int_{\partial U \cap B_{r_k}(x_i)} \phi \:\!h \,\mathcal H^{n-1} \\
    &\geq \frac{\gamma}{c_nr_k} \sum_{i=1}^N \int_{U \cap B_{r_k}(x_i)} \phi \:\!m\,\d x - \frac{2\eps}{c_n}  \sum_{i=1}^N \int_{\partial U \cap B_{r_k}(x_i)} h \,\d \mathcal{H}^{n-1} \\
    &\geq \frac{\gamma}{c_nr_k} \int_{U \cap B_{r_k/2}(\partial U)} \phi \:\!m\,\d x - \frac{2\eps}{c_n^2} \int_{\partial U} h \,\d \mathcal{H}^{n-1}.
    \end{split}
  \end{equation} 
  As this holds for all $k \geq k_0$, sending $k \to \infty$ we infer that
  \begin{equation}
    \int_{\partial U} \phi\:\!h\,\d\mathcal{H}^{n-1} \geq \frac{\gamma}{2c_n}\int_{\partial U} \phi \,\d m_U^{n-1} - \frac{2\eps}{c_n^2} \int_{\partial U} h \,\d \mathcal H^{n-1}.
  \end{equation} 
  Since $\eps>0$ and $\phi$ were arbitrary, the estimate \eqref{eq:hm_measure_comparison} follows.
\end{proof}

Combining this with Lemma~\ref{lem:silhavy_admissibility}, we infer the following result concerning the associated normal traces.
Note that we do not assume that $\mathcal H^{n-1}(\partial U \setminus \rbdry U) = 0$, however we instead require that $h$ is integrable on $\partial U$.

\begin{corollary}
  Given an open set $\Omega \subset \mathbb R^n$, let $\FF \in \mathcal{DM}^1(\Omega)$ with precise majorant $h$, and suppose $U \Subset \Omega$ is a set of finite perimeter such that $\partial U$ satisfies the uniform lower density condition.
  If we have
  \begin{equation}
    \int_{\partial U} h \,\d \mathcal H^{n-1} < \infty,
  \end{equation} 
  then the normal trace of $\FF$ on $\partial U$ is represented by a measure such that
  \begin{equation}
    |(\FF \cdot \nu)_{\partial U}| \leq c(n,\gamma_0) h\,\mathcal{H}^{n-1} \mres \partial U.
  \end{equation} 
\end{corollary}

\begin{remark}
For our arguments, it is essential that $h$ takes the specific form from Definition~\ref{def:precise_majorant}. In particular, it is not known whether the aforementioned results hold if $h$ merely satisfies \eqref{eq:h_property1} and \eqref{eq:h_property2}.
\end{remark}

In comparison, as a consequence of Theorem~\ref{thm:prescribed_trace}, the $\mathcal O_m$-condition does not impose any restriction on the flux.

\begin{theorem}
    Let $U \subset \mathbb R^n$ be a bounded Lipschitz domain, and $\sigma \in \mathcal M(\partial U)$.
    Then there exists a Cauchy flux $\mathcal F$ on $\mathbb R^n$ for which $U \in \mathcal O_{m}$ and we have
    \begin{equation}
      \mathcal F_{U}(S) = \sigma(S) \quad\mbox{for all Borel subsets $S \subset \partial U$.}
    \end{equation}
\end{theorem}

This follows by combining Theorem~\ref{thm:prescribed_trace} and Theorem~\ref{thm:flux_equiv_2}.
Thus, at least when working with regular open sets (Lipschitz regularity is sufficient), the $\mathcal O_m$ class appears to be less restrictive, allowing for cases where the trace can be any measure along $\partial U$.

\appendix
\section{Improved integrability of the potential}\label{sec:appendix_improvement}

In this section we will sketch the proof behind Remark~\ref{rem:improvement_integrability}.
It suffices to establish an analogue in the half-space setting considered in Lemma~\ref{lem:G_potential}, which we formulate as follows:

\begin{lemma}\label{eq:better_p}
  Let $1<p<\infty$ and suppose that $\sigma \in \WW^{-\frac1p,p}(\mathbb R^{n-1})$. Then if we define $\GG$ as in \eqref{eq:G_potential}, we have $\GG \in \LL^p(\mathbb R^{n-1} \times (-T,T))$ for all $T>0$.
\end{lemma}

Once this is established, the rest of the proof of Theorem~\ref{thm:prescribed_trace} can be carried out with this $p$, noting in the final step that the Bogovskii operator is bounded on $\LL^p$ for all $1<p<\infty$.
Here $\WW^{s,p}(\mathbb R^{n-1})$ are understood as the Sobolev--Slobodeckij spaces, which for non-integer $s$ coincide with the Besov spaces $\BB^{s,p}_p(\mathbb R^{n-1})$ (see \cite[Section 2.5.7]{Triebel1983}).
The proof will make use of the Fourier transform, which for $u \in \mathscr{S}(\mathbb R^{n-1})$ is defined as
\begin{equation}
  \hat{u}(\xi) = \int_{\mathbb R^{n-1}} u(x) \mathrm{e}^{-\mathrm{i} \xi \cdot x} \,\d x,
\end{equation} 
extended to $\mathscr{S}'(\mathbb R^{n-1})$ by duality. In particular, $\widehat{\partial_{x_j}u} = \mathrm{i} \xi_j \hat{u}$.

\begin{proof}
  We will represent $\sigma$ as a sum of derivatives of functions in $\WW^{1-\frac1p,p}(\mathbb R^{n-1})$, to which we can employ estimates analogous to the extension results in \cite[Theorem~18.28]{Leoni2017}.

  \textbf{Step 1}: We claim that, if $\sigma \in \WW^{-\frac1p,p}(\mathbb R^{n-1})$, then there exist $f_0,f_1,\ldots,f_{n-1} \in \WW^{1-\frac1p,p}(\mathbb R^{n-1})$ such that
  \begin{equation}\label{eq:sigma_representation}
  \sigma = f_0 + \sum_{j=1}^{n-1} \partial_{x_j} f_j \quad\mbox{as distributions in } \mathbb R^{n-1},
\end{equation} 
with the equivalence of norms 
\begin{equation}
  \lVert \sigma\rVert_{\WW^{-\frac1p,p}(\mathbb R^{n-1})} \simeq \sum_{j=0}^{n-1} \lVert f_j \rVert_{\WW^{1-\frac1p,p}(\mathbb R^{n-1})}.
\end{equation} 

To achieve this, we define $g$ via the Fourier transform as
\begin{equation}
  \hat{g}(\xi) = (1+\lvert \xi\rvert^2)^{-\frac12} \hat{\sigma}(\xi).
\end{equation} 
Then, by \cite[Section 2.4.8]{Triebel1983} we have $g \in \WW^{1-\frac1p,p}(\mathbb R^{n-1})$ with the associated estimate
\begin{equation}
  \lVert g \rVert_{\WW^{1-\frac1p,p}(\mathbb R^{n-1})} \leq C \lVert \sigma \rVert_{\WW^{-\frac1p,p}(\mathbb R^{n-1})}.
\end{equation} 
Moreover, as in \cite[Section 2.4.8]{Triebel1983}, we take Fourier multipliers $\rho_j$ for $\WW^{1-\frac1p,p}(\mathbb R^{n-1})$ such that $1 + \sum_{j=1}^{n-1}\rho_j(\xi)\xi_j \geq (1+\lvert \xi\rvert^2)^{\frac12}$ and define
\begin{equation}
  M(\xi) = \frac{(1+\lvert \xi\rvert^2)^{\frac12}}{1 + \sum_{j=1}^{n-1} \rho_j(\xi) \xi_j},
\end{equation} 
which is also a Fourier multiplier for $\WW^{1-\frac1p,p}(\mathbb R^{n-1})$ by \cite[Section 2.3.7]{Triebel1983}.
More explicitly, as in \cite[Theorem~6.2.3]{BerghLofstrom1976} we can take a cutoff $\chi \in \CC^{\infty}(\mathbb R)$ such that $1_{(-1,1)} \leq 1 - \chi \leq 1_{(-2,2)}$ and set $\rho_i(\xi) = \chi(\xi_j) \frac{\xi_j}{\lvert \xi\rvert}$.

We define $f_0, \cdots, f_{n-1}$ via their Fourier transforms as
\begin{equation}
  \hat{f}_0(\xi) = M(\xi) \hat{g}(\xi), \quad \hat{f}_j(\xi) = \mathrm{i}^{-1} M(\xi) \rho_j(\xi) \hat{g}(\xi) \quad\mbox{for } j = 1,\ldots,n-1.
\end{equation} 
We can then write
\begin{equation}
  \hat{\sigma} = \hat{f}_0 + \sum_{j=1}^{n-1} \mathrm{i}\xi_j\hat{f}_j,
\end{equation}  
which gives the claimed representation \eqref{eq:sigma_representation},
and the above Fourier multiplier estimates imply that
\begin{equation}
  \sum_{i=0}^{n-1} \lVert f_j \rVert_{\WW^{1-\frac1p,p}(\mathbb R^{n-1})} \leq c \lVert g \rVert_{\WW^{1-\frac1p,p}(\mathbb R^{n-1})} \leq c \lVert \sigma \rVert_{\WW^{-\frac1p,p}(\mathbb R^{n-1})},
\end{equation} 
from which the claimed estimate follows.

\textbf{Step 2}: 
Using the representation from the first step, we can write the field $\GG$ from \eqref{eq:G_potential} as
\begin{equation*}
  \begin{split}
    G(y',y_n) &= \int_{\mathbb R^{n-1}} \Phi(y'-z',y_n) f_0(z')\,\d z' - \sum_{j=1}^{n-1} \int_{\mathbb R^{n-1}} \partial_j \Phi(y'-z',y_n) f_j(z') \,\d z' \\
              &=: G_0(y',y_n) + \sum_{j=1}^{n-1} G_j(y',y_n).
  \end{split}
  \end{equation*} 
  We now show that each term lies in $\LL^p$.
  For the first term, we use Young's convolution inequality as in \eqref{eq:G_integrability_estimate} to bound 
  \begin{equation}
    \begin{split}
    \int_{-T}^T \int_{\mathbb R^{n-1}} \lvert G_0(y',y_n)\rvert^p \,\d y'\,\d y_n 
    &\leq c \int_{-T}^T \lVert \Phi(\cdot,y_n)\rVert_{\LL^1(\mathbb R^{n-1})}^p \lVert f_0 \rVert_{\LL^p(\mathbb R^{n-1})}^{p} \,\d y_n \\
    &\leq c T\lVert f_0 \rVert_{\LL^p(\mathbb R^{n-1})}^{p}.
  \end{split}
  \end{equation} 
  Now assume that $1 \leq j \leq n-1$.
  Since $\int_{\mathbb R^{n-1}} \partial_j \Phi(y'-z',y_n) \,\d z' = 0$ and since $\Phi(\cdot,y_n)$ is supported on $B_{\lvert y_n\rvert}^{n-1}(0)$ satisfying $\lvert \partial_j \Phi(y',y_n) \rvert \leq c \lvert y_n \rvert^{-n}$, we infer the pointwise estimate
  \begin{equation}
    \lvert G_j(y',y_n) \rvert \leq \frac{c}{\lvert y_n\rvert^n} \int_{B_{\lvert y_n\rvert }^{n-1}(y')} \lvert f_j(z') - f_j(y')\rvert \,\d z'.
  \end{equation} 
  Hence taking $\LL^p$-norms on both sides and arguing as in \cite[Theorem~18.28]{Leoni2017}, we can bound
  \begin{equation}
    \begin{split}
      \lVert G_j \rVert_{\LL^p(\mathbb R^{n}_+)}^p
      &\leq c \int_{\mathbb R^n_+} \frac1{y_n^{n-1+p}} \int_{B_{y_n}^{n-1}(y')}\lvert f_j(z')-f_j(y')\rvert^p \,\d z' \,\d y \\
      &= c \int_{\mathbb R^{n-1}}\int_{\mathbb R^{n-1}} \lvert f_j(z')-f_j(y')\rvert^p \bigg[\int_{\lvert y'-z'\rvert}^{\infty} \frac1{y_n^{n-1+p}} \,\d y_n\bigg] \,\d z'\,\d y'\\
      &=  \frac{c}{n+p-2}\int_{\mathbb R^{n-1}}\int_{\mathbb R^{n-1}} \frac{|f_j(z')-f_j(y')|^{p}}{\lvert y'-z'\rvert^{n-1+(1-\frac1p)p}} \,\d z' \,\d y',
    \end{split}
  \end{equation} 
  and similarly on $\mathbb R^n_-$ to infer that
  \begin{equation}
    \lVert G_j\rVert_{\LL^p(\mathbb R^n)} \leq c \lVert f_j \rVert_{\WW^{1-\frac1p,p}(\mathbb R^{n-1})}.
  \end{equation} 
    Hence it follows that $G \in \LL^p(\mathbb R^{n-1} \times (-T,T))$, as required.
\end{proof}

\addcontentsline{toc}{section}{Acknowledgements}
\phantomsection
\noindent\subsection*{Acknowledgements}
The authors thank David Wiedemann for insightful discussions on the Piola transform and related topics. AS also thanks Giacomo Del Nin for valuable discussions. Part of this work was carried out while CI was a postdoc at TU Dortmund, and when AS was visiting in May 2025. AS thanks the Max Planck Institute for Mathematics in the Sciences (MiS) in Leipzig for its hospitality and productive research environment.

\printbibliography

@article{AmbrosioColesantiVilla2008,
  title = {Outer {{Minkowski}} Content for Some Classes of Closed Sets},
  author = {Ambrosio, Luigi and Colesanti, Andrea and Villa, Elena},
  date = {2008-12-01},
  journaltitle = {Mathematische Annalen},
  shortjournal = {Math. Ann.},
  volume = {342},
  number = {4},
  pages = {727--748},
  issn = {1432-1807},
  doi = {10.1007/s00208-008-0254-z},
  url = {https://doi.org/10.1007/s00208-008-0254-z},
  urldate = {2025-12-10},
  langid = {english}
}

@book{AmbrosioEtAl2000,
  title = {Functions of Bounded Variation and Free Discontinuity Problems},
  author = {Ambrosio, Luigi and Fusco, Nicola and Pallara, Diego},
  year = {2000},
  pages = {434},
  publisher = {{Clarendon Press}},
  isbn = {978-0-19-850245-6}
}

@article{BanfiFabrizio1979,
  title = {Sul Concetto di Sottocorpo nella Meccanica dei Continui},
  author = {Banfi, Carlo and Fabrizio, Mauro},
  date = {1979},
  journaltitle = {Atti della Accademia Nazionale dei Lincei. Classe di Scienze Fisiche, Matematiche e Naturali. Rendiconti},
  shortjournal = {Rend. Acc. Naz. Lincei},
  volume = {66},
  number = {2},
  pages = {136--142},
  publisher = {Accademia Nazionale dei Lincei},
  issn = {0392-7881},
  url = {https://eudml.org/doc/288621},
  urldate = {2026-03-25},
  langid = {italian}
}

@book{BerghLofstrom1976,
  title = {Interpolation Spaces},
  author = {Bergh, Jöran and Löfström, Jörgen},
  date = {1976},
  series = {Grundlehren Der Mathematischen Wissenschaften},
  volume = {223},
  publisher = {Springer Berlin Heidelberg},
  location = {Berlin, Heidelberg},
  doi = {10.1007/978-3-642-66451-9},
  url = {http://link.springer.com/10.1007/978-3-642-66451-9},
  isbn = {978-3-642-66453-3}
}

@article{Bogovskii1979,
  title = {Solution of the First Boundary Value Problem for the Equation of Continuity of an Incompressible Medium},
  author = {Bogovskii, M. E.},
  date = {1979},
  journaltitle = {Dokl. Akad. Nauk SSSR},
  volume = {248},
  pages = {1037--1040},
  url = {https://cir.nii.ac.jp/crid/1573950400271528192},
  urldate = {2026-01-12}
}

@article{Boussinesq1878,
    author = {Boussinesq., J.},
    title = {{\'{E}}quilibre d’{\'{e}}lasticit{\'{e}} d’un sol isotrope sans pesanteur, supportant diff{\'{e}}rents poids},
    journal = {C. R. Acad. Sci. Paris},
    volume = {86},
    pages = {1260–1263},
    year = {1878}
}

@book{Brezis2011,
  title = {Functional Analysis, {{Sobolev}} Spaces and Partial Differential Equations},
  author = {Brezis, Haim},
  year = {2011},
  publisher = {{Springer New York}},
  address = {{New York, NY}},
  doi = {10.1007/978-0-387-70914-7},
  urldate = {2022-08-07},
%  isbn = {978-0-387-70913-0 978-0-387-70914-7},
  langid = {english}
}

@article{ComiCrastaDeCiccoMalusa2024,
  title = {Representation Formulas for Pairings between Divergence-Measure Fields and \emph{BV} Functions},
  author = {Comi, Giovanni E. and Crasta, Graziano and De Cicco, Virginia and Malusa, Annalisa},
  date = {2024-01-01},
  journaltitle = {Journal of Functional Analysis},
  shortjournal = {J. Funct. Anal.},
  volume = {286},
  number = {1},
  pages = {110192},
  issn = {0022-1236},
  doi = {10.1016/j.jfa.2023.110192},
  url = {https://www.sciencedirect.com/science/article/pii/S002212362300349X},
  urldate = {2025-02-24}
}

@article{ChenFrid1999,
  title = {Divergence‐{{Measure Fields}} and {{Hyperbolic Conservation Laws}}},
  author = {Chen, Gui-Qiang and Frid, Hermano},
  date = {1999-06-01},
  journaltitle = {Archive for Rational Mechanics and Analysis},
  shortjournal = {Arch Rational Mech Anal},
  volume = {147},
  number = {2},
  pages = {89--118},
  issn = {1432-0673},
  doi = {10.1007/s002050050146},
  url = {https://doi.org/10.1007/s002050050146},
  urldate = {2023-11-30},
  langid = {english}
}

@article{ChenFrid2003,
  title = {Extended {{Divergence-Measure Fields}} and the {{Euler Equations}} for {{Gas Dynamics}}},
  author = {Chen, Gui-Qiang and Frid, Hermano},
  date = {2003-05-01},
  journaltitle = {Communications in Mathematical Physics},
  shortjournal = {Commun. Math. Phys.},
  volume = {236},
  number = {2},
  pages = {251--280},
  issn = {1432-0916},
  doi = {10.1007/s00220-003-0823-7},
  url = {https://doi.org/10.1007/s00220-003-0823-7},
  urldate = {2023-11-30},
  langid = {english}
}

@article{ChenComiTorres2019,
    title = {Cauchy Fluxes and Gauss–Green Formulas for Divergence-Measure Fields Over General Open Sets},
    author = {Chen, Gui-Qiang and Comi, Giovanni Eugenio and Torres, Monica}, 
    journaltitle = {Archive for Rational Mechanics and Analysis},
    shortjournal = {Arch. Rational Mech. Anal.},
    volume={233},
    number={1},
    year = {2019},
    pages = {87–166},
    doi = {10.1007/s00205-018-01355-4}
}

@article{ChenIrvingTorres2025,
  title = {Extended {{Divergence-Measure Fields}}, the {{Gauss-Green Formula}} and {{Cauchy Fluxes}}},
  author = {Chen, Gui-Qiang G. and Irving, Christopher and Torres, Monica},
  date = {2025-11-25},
  journaltitle = {Archive for Rational Mechanics and Analysis},
  shortjournal = {Arch Rational Mech Anal},
  volume = {249},
  number = {6},
  pages = {79},
  issn = {1432-0673},
  doi = {10.1007/s00205-025-02135-7},
  url = {https://doi.org/10.1007/s00205-025-02135-7},
  urldate = {2025-12-10},
  langid = {english}
}

@article{ChenLiTorres2020,
% ISSN = {00222518, 19435258},
 URL = {https://www.jstor.org/stable/26959170},
 author = {Gui-Qiang Chen and Qinfeng Li and Monica Torres},
 journaltitle = {Indiana University Mathematics Journal},
 shortjournal={Indiana Univ. Math. J.},
 number = {1},
 pages = {pp. 229--264},
 publisher = {Indiana University Mathematics Department},
 title = {Traces and Extensions of Bounded Divergence-Measure Fields on Rough Open Sets},
 urldate = {2024-11-20},
 volume = {69},
 year = {2020},
 doi = {10.1512/iumj.2020.69.8375}
}

@article{CrastaDeCicco2019,
  title = {Anzellotti's Pairing Theory and the Gauss–Green Theorem},
  author = {Crasta, Graziano and De Cicco, Virginia},
  date = {2019-02-05},
  journaltitle = {Advances in Mathematics},
  shortjournal = {Adv. Math.},
  volume = {343},
  pages = {935--970},
  issn = {0001-8708},
  doi = {10.1016/j.aim.2018.12.007},
  url = {https://www.sciencedirect.com/science/article/pii/S0001870818304997},
  urldate = {2025-02-24}
}

@article{ChenTorres2005,
  title = {Divergence-Measure Fields, Sets of Finite Perimeter, and Conservation Laws},
  author = {Chen, Gui-Qiang and Torres, Monica},
  date = {2005-02-01},
  journaltitle = {Archive for Rational Mechanics and Analysis},
  shortjournal = {Arch. Rational Mech. Anal.},
  volume = {175},
  number = {2},
  pages = {245--267},
  issn = {1432-0673},
  doi = {10.1007/s00205-004-0346-1},
  url = {https://doi.org/10.1007/s00205-004-0346-1},
  urldate = {2023-11-30},
  langid = {english}
}

@article{ChenTorresZiemer2009,
author = {Chen, Gui-Qiang and Torres, Monica and Ziemer, William P.},
title = {Gauss-Green theorem for weakly differentiable vector fields, sets of finite perimeter, and balance laws},
journaltitle = {Communications on Pure and Applied Mathematics},
shortjournal={Comm. Pure Appl. Math.},
volume = {62},
number = {2},
pages = {242-304},
doi = {10.1002/cpa.20262},
url = {https://onlinelibrary.wiley.com/doi/abs/10.1002/cpa.20262},
%eprint = {https://onlinelibrary.wiley.com/doi/pdf/10.1002/cpa.20262},
year = {2009}
}

@article{DegiovanniEtAl1999,
  title = {Cauchy {{Fluxes Associated}} with {{Tensor Fields Having Divergence Measure}}},
  author = {Degiovanni, Marco and Marzocchi, Alfredo and Musesti, Alessandro},
  year = {1999},
  month = aug,
  journaltitle = {Archive for Rational Mechanics and Analysis},
  shortjournal = {Arch. Rational Mech. Anal.},
  volume = {147},
  number = {3},
  pages = {197--223},
  issn = {1432-0673},
  doi = {10.1007/s002050050149},
  urldate = {2023-11-30},
  langid = {english}
}

@book{Federer1996,
  title = {Geometric Measure Theory},
  author = {Federer, Herbert},
  editor = {Eckmann, B. and family=Waerden, given=B. L., prefix=van der, useprefix=true},
  date = {1996},
  series = {Classics in Mathematics},
  publisher = {Springer Berlin Heidelberg},
  location = {Berlin, Heidelberg},
  doi = {10.1007/978-3-642-62010-2},
  url = {http://link.springer.com/10.1007/978-3-642-62010-2},
  isbn = {978-3-540-60656-7}
}

@article{Flamant1892,
    author = {Flamant, A.},
    title = {Sur la r\'epartition des pressions dans un solide rectangulaire charg\'e transversalement.},
    journal = {C. R. Acad. Sci. Paris},
    volume = {114},
    pages = {1465–1468},
    year = {1892}
}

@article{GurtinMartins1976,
  title = {Cauchy's Theorem in Classical Physics},
  author = {Gurtin, Morton E. and Martins, Luiz C.},
  date = {1976-12-01},
  journaltitle = {Archive for Rational Mechanics and Analysis},
  shortjournal = {Arch. Rational Mech. Anal.},
  volume = {60},
  number = {4},
  pages = {305--324},
  issn = {1432-0673},
  doi = {10.1007/BF00248882},
  url = {https://doi.org/10.1007/BF00248882},
  urldate = {2025-09-17},
  langid = {english}
}

@online{Irving2025,
  title = {On the Normal Trace Space of Extended Divergence-Measure Fields},
  author = {Irving, Christopher},
  date = {2025-03-12},
  eprint = {2503.09536},
  eprinttype = {arXiv},
  eprintclass = {math},
  doi = {10.48550/arXiv.2503.09536},
%  url = {http://arxiv.org/abs/2503.09536},
  pubstate = {Preprint.}
}

@article{Kraft2016,
    author = {Kraft, Daniel},
    title =   {Measure-Theoretic Properties of Level Sets of Distance Functions},
    journaltitle = {Journal of Geometric Analysis},
    shortjournal = {J. Geom. Anal.},
    volume = {26},
    year = {2016},
    pages = {2777–2796},
    doi = {10.1007/s12220-015-9648-9}
}

@book{Leoni2017,
  title = {A First Course in {{Sobolev}} Spaces},
  author = {Leoni, Giovanni},
  date = {2017},
  series = {Graduate Studies in Mathematics},
  edition = {2nd ed},
  number = {181},
  publisher = {AMS},
  location = {Providence},
  isbn = {978-1-4704-2921-8},
  langid = {english}
}

@article{MarzocchiMusesti2001,
    author = {Marzocchi, A. and Musesti, A.},
    title = {Decomposition and integral representation of Cauchy interactions associated with measures},
    journal = {Continuum Mech. Thermodyn.},
    volume = {13},
    pages = {149–169},
    year = {2001},
    url = {https://doi.org/10.1007/s001610100046}
}

@article{MitreaTaylor2000,
  title = {Potential Theory on Lipschitz Domains in Riemannian Manifolds: Sobolev–Besov Space Results and the Poisson Problem},
  shorttitle = {Potential Theory on Lipschitz Domains in Riemannian Manifolds},
  author = {Mitrea, Marius and Taylor, Michael},
  date = {2000-09-10},
  journaltitle = {Journal of Functional Analysis},
  shortjournal = {Journal of Functional Analysis},
  volume = {176},
  number = {1},
  pages = {1--79},
  issn = {0022-1236},
  doi = {10.1006/jfan.2000.3619},
  url = {https://www.sciencedirect.com/science/article/pii/S002212360093619X},
  urldate = {2026-04-22}
}

@incollection{Noll1959,
  title = {The {{Foundations}} of {{Classical Mechanics}} in the {{Light}} of {{Recent Advances}} in {{Continuum Mechanics}}},
  booktitle = {Studies in {{Logic}} and the {{Foundations}} of {{Mathematics}}},
  author = {Noll, Walter},
  editor = {Henkin, Leon and Suppes, Patrick and Tarski, Alfred},
  date = {1959-01-01},
  series = {The {{Axiomatic Method}}},
  volume = {27},
  pages = {266--281},
  publisher = {Elsevier},
  doi = {10.1016/S0049-237X(09)70033-3},
  url = {https://www.sciencedirect.com/science/article/pii/S0049237X09700333},
  urldate = {2025-09-17}
}

@article{Podio-Guidugli2005,
    author = {Podio-Guidugli, Paolo},
    title = {Examples of concentrated contact interactions in simple bodies.},
    journal = {J. Elasticity},
    volume = {75},
    number = {(2)},
    pages = {167–186},
    doi = {10.1007/s10659-005-3029-8},
    year = {2005}
}

@article{Podio-GuidugliSchuricht2012,
    author = {Podio-Guidugli, P. and Schuricht, F.},
    title = {Concentrated actions on cuspidate plane bodies.},
    journal = {J. Elasticity},
    volume = {106},
    number = {no. 2},
    pages = {107–114},
    year = {2012},
    doi = {10.1007/s10659-010-9295-0}
}

@book{SchurichtSchonherr2025,
  title = {A {{Theory}} of {{Traces}} and the {{Divergence Theorem}}},
  author = {Schuricht, Friedemann and Schönherr, Moritz},
  date = {2025},
  series = {Lecture {{Notes}} in {{Mathematics}}},
  volume = {2372},
  publisher = {Springer Nature Switzerland},
  location = {Cham},
  doi = {10.1007/978-3-031-86664-7},
  url = {https://link.springer.com/10.1007/978-3-031-86664-7},
  urldate = {2025-09-17},
  isbn = {978-3-031-86663-0 978-3-031-86664-7},
  langid = {english}
}

@article{Schuricht2007a,
  title = {A {{New Mathematical Foundation}} for {{Contact Interactions}} in {{Continuum Physics}}},
  author = {Schuricht, Friedemann},
  year = {2007},
  month = jun,
  journaltitle = {Archive for Rational Mechanics and Analysis},
  shortjournal = {Arch. Rational Mech. Anal.},
  volume = {184},
  number = {3},
  pages = {495--551},
  issn = {1432-0673},
  doi = {10.1007/s00205-006-0032-6},
  urldate = {2023-11-30},
  langid = {english}
}

@inbook{Schuricht2008,
  title = {Interactions in Continuum Physics},
  booktitle = {Mathematical Modelling of Bodies with Complicated Bulk and Boundary Behavior},
  author = {Schuricht, Friedemann},
  date = {2008},
  series = {Quaderni Di Matematica},
  volume = {20},
  pages = {169--196},
  publisher = {Aracne},
  location = {Napoli},
  bookauthor = {Šilhavý, Miroslav},
  isbn = {978-88-548-1956-6}
}

@article{Silhavy1985,
  title = {The Existence of the Flux Vector and the Divergence Theorem for General {{Cauchy}} Fluxes},
  author = {Šilhavy, Miroslav},
  date = {1985-09-01},
  journaltitle = {Archive for Rational Mechanics and Analysis},
  shortjournal = {Arch. Rational Mech. Anal.},
  volume = {90},
  number = {3},
  pages = {195--212},
  issn = {1432-0673},
  doi = {10.1007/BF00251730},
  url = {https://doi.org/10.1007/BF00251730},
  urldate = {2025-09-17},
  langid = {english}
}

@article{Silhavy1991,
  title = {Cauchy's Stress Theorem and Tensor Fields with Divergences in {{$L^p$}}},
  author = {{\v S}ilhav{\'y}, M.},
  year = {1991},
  month = sep,
  journaltitle = {Archive for Rational Mechanics and Analysis},
  shortjournal = {Arch. Rational Mech. Anal.},
  volume = {116},
  number = {3},
  pages = {223--255},
  issn = {1432-0673},
  doi = {10.1007/BF00375122},
  urldate = {2023-11-30},
  langid = {english}
}

@article{Silhavy2005,
    title = {Divergence Measure Fields and Cauchy's Stress Theorem},
    author = {{\v S}ilhav{\'y}, M.},
    shortjournal = {Rend. Sem. Mat. Univ. Padova},
    volume = {113},
    year = {2005},
    pages = {15--45}
}

@article{Silhavy2009,
  title = {The {{Divergence Theorem}} for {{Divergence Measure Vectorfields}} on {{Sets}} with {{Fractal Boundaries}}},
  author = {{\v S}ilhav{\'y}, M.},
  year = {2009},
  month = jul,
  journaltitle = {Mathematics and Mechanics of Solids},
  volume = {14},
  number = {5},
  pages = {445--455},
  publisher = {{SAGE Publications Ltd STM}},
  issn = {1081-2865},
  doi = {10.1177/1081286507081960},
  urldate = {2023-11-30},
  langid = {english}
}

@book{Triebel1983,
  title = {Theory of Function Spaces},
  author = {Triebel, Hans},
  date = {1983},
  publisher = {Springer Basel},
  location = {Basel},
  doi = {10.1007/978-3-0346-0416-1},
  url = {http://link.springer.com/10.1007/978-3-0346-0416-1},
  isbn = {978-3-0346-0415-4}
}

@article{Villa2009,
  title = {On the Outer {{Minkowski}} Content of Sets},
  author = {Villa, Elena},
  date = {2009-09-01},
  journaltitle = {Annali di Matematica Pura ed Applicata},
  shortjournal = {Annali di Matematica},
  volume = {188},
  number = {4},
  pages = {619--630},
  issn = {1618-1891},
  doi = {10.1007/s10231-008-0093-2},
  url = {https://doi.org/10.1007/s10231-008-0093-2},
  urldate = {2025-12-10},
  langid = {english}
}

@book{Ziemer1989,
  title = {Weakly Differentiable Functions},
  author = {Ziemer, William P.},
  date = {1989},
  series = {Graduate {{Texts}} in {{Mathematics}}},
  volume = {120},
  publisher = {Springer New York},
  location = {New York, NY},
  doi = {10.1007/978-1-4612-1015-3},
  url = {http://link.springer.com/10.1007/978-1-4612-1015-3},
  urldate = {2021-11-24},
  isbn = {978-1-4612-6985-4 978-1-4612-1015-3}
}

\end{document}